\documentclass{article}
\usepackage{amssymb}
\usepackage{amsmath}
\usepackage{amsthm}
\usepackage[english]{babel}
\usepackage{authblk}
\usepackage{graphicx}
\usepackage{float}
\usepackage{enumitem}
\usepackage{tikz}
\usetikzlibrary{matrix, calc}
\usetikzlibrary{fit}
\usetikzlibrary{decorations.pathreplacing}
\usepackage{dsfont}
\usepackage{xcolor}
\usepackage{mathtools}
\numberwithin{figure}{section}

\newtheorem{thm}{Theorem}[section]
\newtheorem{lem}[thm]{Lemma}

\newtheorem{cor}[thm]{Corollary}
\newtheorem{conjecture}[thm]{Conjecture}
\newtheorem{defn}[thm]{Definition}

\newcommand{\N}{\mathbb{N}}

\newcommand{\norm}[1]{\left\| #1 \right\|}

\newcommand{\diag}{\mathop{\mathrm{diag}}}\newcommand{\esssup}{\mathop{\mathrm{ess\,sup}}}\newcommand{\essinf}{\mathop{\mathrm{ess\,inf}}}\renewcommand{\i}{\textup{i}}\newcommand{\bfc}{\mathbf c}

\DeclareMathSymbol{\shortminus}{\mathbin}{AMSa}{"39}

\providecommand{\keywords}[1]{\textit{Keywords } #1}

\usepackage{graphicx} 

\title{A note on the spectral distribution of non-Hermitian block matrices with Toeplitz blocks}
\author{Andrea Adriani$^{(1)}$ and Giacomo Tento$^{(2)}$}
\date{}

\begin{document}

\maketitle

\section*{Abstract}

In the present paper, we are concerned with the study of the spectral distribution of matrix-sequences  showing a non-Hermitian block structure with Toeplitz blocks. We use the notion of geometric mean of matrices and the theory of Generalized Locally Toeplitz (GLT) sequences to perform our analysis and produce some numerical tests and visualizations to confirm our theoretical derivations.

\ \\
\ \\
\medskip

\noindent
$(1)$\, Vanguard Center, University Mohammed VI Polytechnic, Morocco\\
(andrea.adriani@um6p.ma, ORCID ID: 0000-0003-3390-7891);\\
$(2)$\, University of Insubria, Via Valleggio 11, 22100, Como, Italy\\
(gtento@uninsubria.it,  ORCID ID:0009-0000-5773-7829)

\ \\
\ \\
\medskip
\noindent
\keywords{Block matrices, Toeplitz matrices, Spectral distribution}

\section{Introduction}\label{sec:intro}

A matrix-sequence $\left\lbrace A_n \right\rbrace_n $ is an ordered collection of matrices $A_n \in \mathbb{C}^{d_n \times d_n}$, with $d_n$ strictly increasing and $n \in \N$. 
Such sequences naturally arise from the numerical discretizaton of integral, ordinary and partial differential equations; most often they possess a Generalized Locally Toeplitz (GLT) symbol (see, for instance, \cite{G26}), that is, a measurable function $f: \Omega \subset \mathbb{R} \rightarrow \mathbb{C}^{s \times s} $ that, under additional circumstances, is able to provide a description of the asymptotic distribution of the eigenvalues of the sequence. Namely, we say that $\{A_{n}\}_{n}$ is distributed in the eigenvalue sense as $f$ over $\Omega$ and we write $\{A_{n}\}_{n} \sim_{\lambda} (f,\Omega)$, if
\begin{equation}\label{intro_distr}
    \lim_{n \to \infty} \sum_{i=1}^{d_n} F\left( \lambda_i \left( A_n \right)\right) = \frac{1}{\mathcal{L} \left( \Omega \right)} \int_{\Omega} \frac{\sum_{j=1}^{s} F \left( \lambda_j(f (x))\right)}{s}  \mathrm{d} x 
\end{equation}
for every continuous function $F: \mathbb{C} \rightarrow \mathbb{C}$ with compact support. \\
For our purposes, we are interested in uni-variate matrix-valued symbols, but the notion of spectral distribution can been extended also to multilevel matrix-sequences and multilevel block matrix-sequences with multi-variate matrix-valued symbols (see \cite{book-glt-2,GMS18}). \\
The theory of GLT sequences stems from previous inquiries on Locally Toeplitz (LT) sequences (\cite{Tilli}), and underwent a systematic study in \cite{book-glt-1}. \\
Notable applications of this theoretical apparatus comprise finite differences methods (see, e.g., \cite{SS03}), finite elements methods (\cite[Section 10.6]{book-glt-1}), finite volumes (\cite{B17}) and fractional derivatives (\cite{IKLST26}). \\
{Moreover, non-symmetric matrices with block structure arise naturally in a wide range of applications, particularly in coupled problems where different physical quantities are linked via single PDE system. Examples include computational fluid dynamics, diffusion-reaction equations, magnetohydrodynamics, poroelasticity, and electro-diffusion models. Such problems must frequently be solved at large scale, especially in realistic multi-dimensional settings, assembling systems with up to billions of degrees of freedom to accurately resolve the coupling of the different physical quantities. In this context of real-world applications, specialized tools from both numerical linear algebra and GLT theory are essential, both for analyzing the theoretical properties of known solution strategies as stated above and to design tailored numerical algorithms with high efficiency and scalability.}
\ \\
Furthermore, information provided by means of GLT sequences are not only of interest from a theoretical standpoint, but often carries significant practical purposes. 
For instance, it is known that the convergence rate of Krylov methods strongly depends on the spectral features of the matrices to which they are applied; thus, the spectral distribution notion can be employed to predict and improve their performances via preconditioning techniques (see \cite[Section 10]{book-glt-1}).

In this paper, we consider sequences of block matrices $k \times k$, $k\geq 2$, with the following structure:

\begin{equation}\label{eq:A_Toeplitz}
    \mathcal{A}_n = 
    \begin{bmatrix}
T_n(f_{1,1}) & T_{n}(f_{1,2}) &  &  \\
T_n(f_{2,1}) & T_n(f_{2,2}) & \ddots &  \\
 & \ddots & \ddots & T_n(f_{k-1 , k}) \\
 &  & T_n(f_{k, k-1}) & T_n(f_{k, k})
\end{bmatrix},
\end{equation}
where each $f_{i,j}$ is a {strictly positive} function in $L^{\infty}([-\pi,\pi])$ and $T_n(f)$ denotes the Toeplitz matrix of dimension $n$ generated by $f$. 

In the case where the sequence is Hermitian, it is possible to use the theory developed in \cite{AFGS26,AGS26} (see also \cite{BGMS22}) to perform a complete spectral analysis of the sequence, both from the singular value and the eigenvalue perspective. Moreover, even when the sequence is not Hermitian, and the ratio of the dimensions of the various blocks with the global dimension tends (as $n \to \infty$) to any rational or irrational number in $(0,1)$, the tools in \cite{AFGS26,AGS26} allow us to find the singular value distribution of the sequence. However, in the case where the sequence does not present a Hermitian structure, the theory does not explicitly give the eigenvalue distribution, and a new approach is needed to tackle the problem. In this paper, we use the notion of geometric means  of matrices to solve this issue for a class of matrices as in \eqref{eq:A_Toeplitz}. Note that sequences of matrices presenting a similar non-Hermitian structure are already known in applications (see, for example, \cite{BEHR24}, where the authors study electro-diffusion equations).

One of the first formulations of geometric means of matrices can be found in \cite{KA80}, where many properties are discussed. We make use of their definition, that is, for two Hermitian positive definite (HPD) matrices $A$ and $B$, we let
\begin{equation}\label{def:geom_means}
    G(A,B):= A^{1/2} \left( A^{-1/2} B A^{-1/2} \right)^{1/2} A^{1/2}=G(B,A)
\end{equation}
be their geometric mean. For a recent account on how the Generalized Locally Toeplitz (GLT) theory is related to this topic, we refer the reader to \cite{IKLS25}.

The main contribution of the present work is the proof of the following theorem.

\begin{thm}\label{main_thm}
    Let $\mathcal{A}_n$ be as in Equation \eqref{eq:A_Toeplitz}. Assume that $f_{i,j} \in L^{\infty}([-\pi,\pi])$ and real-valued almost everywhere for every $i,j$. Moreover, assume that $\essinf f_{i, i+1},\essinf f_{i+1, i} > 0$ for every $i=1,\dots,k-1$. Then
    \begin{equation*}
        \{\mathcal{A}_n\}_n \sim_{\lambda} F(\theta),
    \end{equation*}
    with
    \begin{equation*}
        F(\theta)= \begin{bmatrix}
f_{1,1} & f_{1,2} &  &  \\
f_{2,1} & f_{2,2} & \ddots &  \\
 & \ddots & \ddots & f_{k-1 , k} \\
 &  & f_{k, k-1} & f_{k, k}
\end{bmatrix}.
    \end{equation*}
    Equivalently,
    \begin{equation*}
        \{\mathcal{A}_n\}_n \sim_{\lambda} \tilde F(\theta),
    \end{equation*}
with
    \begin{equation*}
        \tilde F(\theta) = \begin{bmatrix}
        f_{1,1} & f_{1,2}^{1/2}f_{2,1}^{1/2} &  &  \\
f_{1,2}^{1/2}f_{2,1}^{1/2} & f_{2,2} & \ddots &  \\
 & \ddots & \ddots & f_{k-1 , k}^{1/2} f_{k, k-1}^{1/2} \\
 &  & f_{k , k-1}^{1/2} f_{k-1 , k}^{1/2} & f_{k, k}
\end{bmatrix}.
    \end{equation*}
\end{thm}

In order to prove our main theorem, we will make use of a fine approximation of each Toeplitz block with a circulant matrix depending on the generating function (see \cite{WZ17} and references therein) and of \cite[Theorem 1]{BS20}.

{Note that our main theorem is equivalent to stating that, under suitable assumptions, $\{T_{n}(F)\} \sim_{\lambda} F$ even in the non-Hermitian setting. More general results can be found in \cite{DNS12} with a different technique relying on Mergelyan's theorem. However, their approach is limited to the essentially bounded setting, while our circulant approximation appears to be suitable for dealing with the $L^2$ case, at least. Moreover, our approach is more practical and the approximation technique employed in the current work also suggests preconditioners for large linear systems involving matrices as in \eqref{eq:A_Toeplitz}, whose efficiency we numerically analyze.} 
The paper is organized as follows. In Section \ref{sec:tools} we recall the concepts of spectral distribution, Toeplitz sequences, GLT sequences (and related properties). There, we also recall the definition of geometric mean of two matrices that we use throughout the paper and discuss the correlation between GLT theory and geometric means. Section \ref{sec:main} is the core of the present note and is devoted to the proof of our main result. In Section \ref{sec:test} we perform some numerical tests that corroborate the theory developed {and we propose a preconditioning strategy for solving large linear systems involving matrices as in \eqref{eq:A_Toeplitz}}. Finally, Section \ref{sec:end} is devoted to draw conclusions, provide a discussion on possible improvements, and state some open problems for future research.

\section{Preliminary tools: Spectral distribution, Toeplitz matrices, GLT sequences}\label{sec:tools}

We start this section with the definition of spectral distribution, which is the main objective of investigation in this work. After this, we recall the definition of approximating class of sequences (a.c.s.) and list some classes of meaningful sequences, ending with the class of GLT sequences.

\begin{defn}\label{def-distribution}
Let $\{A_n\}_n$ be a matrix-sequence, with $A_n$ of size $d_n$ increasing, and let $f:\Omega\subset\mathbb R\to\mathbb{C}^{s \times s}$ be
a measurable function defined on a set $\Omega$ with Lebesgue measure $0<\mathcal{L}(\Omega)<\infty$.
\begin{itemize}
    \item We say that $\{A_n\}_n$ has a (asymptotic) singular value distribution described by $f$ on $\Omega$, and we write $\{A_n\}_n\sim_\sigma (f,\Omega)$, if
    \begin{equation}\label{distribution:sv-sv}
     \lim_{n\to\infty}\frac1{d_n}\sum_{i=1}^{d_n}F(\sigma_i(A_n))=\frac{1}{\mathcal{L}(\Omega)}\int_{\Omega} \frac{\sum_{j=1}^s F(\sigma_{j}(f(x)))}{s}{\rm d} x,\qquad\forall\,F\in C_c(\mathbb R).
    \end{equation}
    \item We say that $\{A_n\}_n$ has a (asymptotic) spectral (or eigenvalue) distribution described by $f$ on $\Omega$, and we write $\{A_n\}_n\sim_\lambda (f,\Omega)$, if
    \begin{equation}\label{distribution:sv-eig}
     \lim_{n\to\infty}\frac1{d_n}\sum_{i=1}^{d_n}F(\lambda_i(A_n))=\frac{1}{\mathcal{L}(\Omega)}\int_{\Omega}\frac{\sum_{j=1}^s F(\lambda_{j}(f(x)))}{s} {\rm d} x,\qquad\forall\,F\in C_c(\mathbb C),
    \end{equation}
\end{itemize}
where $C_c(\mathbb{R})$ and $C_c(\mathbb{C})$ represent the sets of continuous and compactly supported functions on $\mathbb{R}$ and $\mathbb{C}$, respectively.
If $\{A_n\}_n$ has both a singular value and an eigenvalue distribution described by $f$, we write $\{A_n\}_n\sim_{\sigma,\lambda}\left(f, \Omega\right)$.
\end{defn}

\begin{defn}
Let $\{A_{n}\}_n$ be a square matrix-sequence of size $d_n$, such that $ d_n \nearrow \infty $, and let $\left\{\{B_{n,t}\}_n\right\}_{t}$ be a sequence of matrix-sequences of the same size $d_n$. We say that $\left\{\{B_{n,t}\}_n\right\}_{t}$ is an approximating class of sequences (a.c.s.) for $\{A_{n}\}_n$ if the following condition is met: for every $t$ there exists $n_t$ such that, for $n>n_t$,
\begin{equation*} \nonumber
A_n =  B_{n,t} + R_{n,t} + N_{n,t},
\end{equation*}
and
\begin{equation*}
    \textnormal{rank}\left(R_{n,t}\right) \leq c(t) d_n,
\end{equation*}
$$
\left\|N_{n,t}\right\|\leq \omega(t),
$$
where $\|\cdot\|$ denotes the spectral norm of a matrix, $n_t, c(t), \omega(t)$ depend only on $t$, and
$$
\lim_{t \to \infty} c(t) = \lim_{t \to \infty} \omega(t) =0.
$$
\end{defn}

The following theorem constitutes a link between the notions of a.c.s. and of spectral distribution.

 \begin{thm}
 \label{acs}\rm(\cite{Tilli})
 Let $\{A_n\}_n$ be a square matrix-sequence of size $d_n$, with $ d_n \nearrow \infty $, and let $\left\{ \{B_{n,t} \}_n \right\}_t$ be an a.c.s. for $ \{ A_n \}_n$. Suppose that $ \{ B_{n,t}\}_n \sim_{\sigma} (f_t, \Omega) $ and $ f_t \to f $ in measure, then $ \{ A_n \}_n \sim_{\sigma} (f, \Omega) $. Furthermore, if all the matrices involved are Hermitian, $ \{B_{n,t} \}_n \sim_{\lambda} (f_t, \Omega) $ and $ f_t \to f $ in measure, then $ \{ A_n \}_n \sim_{\lambda} (f,\Omega) $.
\end{thm}

\subsection{Toeplitz matrices, Diagonal samplings and zero distributed sequences}

In this section we introduce some fundamental classes of sequences for which the spectral distribution is known and which constitute the starting elements for constructing the GLT algebra, namely, Toeplitz sequences, diagonal sampling sequences and zero-distributed sequences.

Toeplitz matrices are matrices of the form $T_n=(t_{i-j})_{i,j=1}^{n}$, thus having constant diagonals. An interesting subclass of Toeplitz matrices is the one of Toeplitz matrices generated by functions in $L^1([-\pi,\pi])$. We use the notation $T_n(f)=(\tilde{f}_{i-j})_{i,j=1}^{n}$ to indicate the $n$-th Toeplitz matrix generated by $f$. The terms $\tilde{f}_k$ represent the Fourier coefficients of the function $f$, namely,
\begin{equation*}
    \tilde{f}_k= \frac{1}{2 \pi} \int_{-\pi}^{\pi} f(\theta) \rm{e}^{-ik\theta } \rm{d}\theta,
\end{equation*}
for all $k \in \mathbb{Z}$. The link between the concepts of spectral distribution and Toeplitz matrices is given by the following result, see \cite{Tilli}.
\begin{thm}
    Let $f \in L^1([-\pi,\pi])$. Then
    \begin{equation*}
        \{T_n(f)\}_n \sim_{\sigma} (f,[-\pi,\pi]).
    \end{equation*}
    If, moreover, the function $f$ is real-valued a.e., then
    \begin{equation*}
        \{T_n(f)\}_n \sim_{\lambda} (f,[-\pi,\pi]).
    \end{equation*}
\end{thm}

Now, let $a: [0,1] \to \mathbb{C}^{r \times r}$. For every $n \in \mathbb{N}$, we define the $n$-th diagonal sampling matrix generated by $a$ as
\begin{equation*}
    D_{n}(a) =\diag_{i=1,\dots,n} a \left( \frac{i}{n} \right).
\end{equation*}
Trivially, we have
\begin{equation*}
    \{D_{n}(a)\}_{n} \sim_{\lambda,\sigma} (a,[0,1]).
\end{equation*}
Finally, the class of zero distributed sequences is the class of sequences $\{Z_n\}_n$ such that
\[
\{Z_n\}_n \sim_{\sigma} 0.
\]
Let $\|\cdot \|_p$ denote the Schatten $p$-norm for every $p\geq 1$ (for $p=\infty$ we have $\|\cdot\|_\infty=\|\cdot\|$, i.e., the Schatten $\infty$-norm is the spectral norm). The following theorem constitutes a key characterization for the class of zero-distributed sequences (see \cite{book-glt-1,SS25}).

\begin{thm}
    Let $\{Z_n\}_n$ be a matrix-sequence with $Z_n$ of size $d_n$. With the natural convention that $1/\infty=0$, we have
    \begin{enumerate}
        \item $\{Z_n\}_n \sim_{\sigma} 0$ if and only if $Z_n=R_n+N_n$, with $\text{rank}(R_n)/d_n \to 0$ and $\|N_n\| \to 0$ as $n \to \infty$;
        \item $\{Z_n\}_n \sim_{\sigma} 0$ if there exists $p \in [1,\infty]$ such that
        \begin{equation*}
            \frac{\|Z_n\|_p}{(d_n)^{1/p}} \to 0 \text{ as } n \to \infty.
        \end{equation*}
    \end{enumerate}
\end{thm}

\subsection{GLT sequences and Geometric means of matrices}\label{ssec:geom_means}

We are now ready to introduce the class of GLT sequences.

A Generalized Locally Toeplitz (GLT) sequence $\{A_{n}\}_n$ is a matrix-sequence equipped with a measurable function $f:[0,1] \times [-\pi,\pi] \to \mathbb{C}$ called GLT symbol. We use the notation $\{A_{n}\}_{n} \sim_{\text{GLT}} f$ to indicate that $\{A_{n}\}_{n}$ is a GLT sequence with symbol $f$. Here, we report the main properties of the $*$-algebra of GLT sequences. The definition of GLT sequences can be rather cumbersome and requires the introduction of many accessory tools. For this reason, we only state the main algebraic properties of GLT sequences. For a complete and recent account of the theory of GLT theory (also in the multi-level setting), we refer the reader to the two comprehensive books \cite{book-glt-1, book-glt-2} and to \cite{GLT-block1D,GLT-blockdD}.

\begin{enumerate}[label=\textbf{GLT \arabic*}]
\setcounter{enumi}{-1}
\item \label{GLT0} If $\{A_n\}_{n} \sim_{\text{GLT}} f$ then $\{A_n\}_n \sim_{\text{GLT}} g$ if and only if $f=g$ a.e.\\
If $f:[0,1] \times[-\pi,\pi] \to \mathbb{C}$ is measurable then there exists $\{A_n\}_n$ such that $\{A_n\}_n \sim_{\text{GLT}} f$;
\item \label{GLT1} if $\{A_{n}\}_{n} \sim_{\text{GLT}} f$ then $\{A_{n}\}_{n}\sim_{\sigma} \left(f,[0,1]\times[-\pi,\pi]\right)$. If $\{A_{n}\}_{n} \sim_{\text{GLT}} f$ and each $A_{n}$ is Hermitian then $\{A_{n}\}_{n}\sim_{\lambda} \left(f,[0,1]\times[-\pi,\pi]\right)$;
\item \label{GLT2} If $\{A_{n}\}_{n} \sim_{\text{GLT}} f$ and $A_{n}=X_{n}+ Y_{n}$, where
\begin{itemize}
    \item every $X_{n}$ is Hermitian,
    \item $\| Y_{n} \|_2/\sqrt{n} \to 0$ as $n \to \infty$,
\end{itemize}
then $\{A_{n}\}_{n} \sim_{\lambda} \left(f,[0,1] \times [-\pi,\pi]\right)$.
\item \label{GLT3} We have
\begin{itemize}
    \item $\{ T_{n}(g)\}_{n} \sim_{\text{GLT}} f(x,\theta)=g(\theta)$ if $g \in L^{1}\left([-\pi,\pi]\right)$;
    \item $\{ D_{n}(a)\}_{n} \sim_{\text{GLT}} f(x,\theta)=a(x)$ if $a$ is a continuous almost everywhere function;
    \item $\{ Z_{n}\}_{n} \sim_{\text{GLT}} f(x,\theta)=0$ if and only if $\{ Z_{n}\}_{n} \sim_{\sigma} 0$;
    \end{itemize}
    \item \label{GLT4}  if $\{A_{n}\}_{n}\sim_{\text{GLT}} f$, then
    \begin{itemize}
        \item $\{A^{*}_{n}\}_{n} \sim_{\text{GLT}} \bar{f}$;
        \item if, additionally, $f\neq 0$ a.e., then $\{A^{\dagger}_{n}\}_{n}\sim_{\text{GLT}} f^{-1}$;
    \end{itemize}
    \item \label{GLT5} if $\{A_{n}\}_{n}\sim_{\text{GLT}} f$ and $\{B_{n}\}_{n}\sim_{\text{GLT}} h$, then
    \begin{itemize}
        \item $\{\alpha A_{n}+\beta B_{n}\}_{n}  \sim_{\text{GLT}} \alpha f+\beta h$ for every $\alpha, \beta \in \mathbb{C}$;
        \item $\{A_{n} B_{n}\}_{n}  \sim_{\text{GLT}}  f h$;
    \end{itemize}
    
    \item \label{GLT6} $\{A_{n}\}_{n} \sim_{\text{GLT}} f$ if and only if there exist GLT sequences $\{B_{n,t}\}_{n} \sim_{\text{GLT}} f_t$ such that $\{\{B_{n,t}\}_{n}\}_t$ is an a.c.s. for $\{A_{n}\}_{n}$ and $f_t \to f$ in measure.
\end{enumerate}

We now recall the concept of geometric mean for two Hermitian positive definite matrices and connect this to the GLT theory. Since we are mainly interested in the spectral properties of the geometric mean of two matrix-sequences, we only give some results in this context and refer the reader to \cite{KA80} for additional details and further properties.

Given two Hermitian Positive Definite (HPD) matrices $A$ and $B$, their geometric mean is defined as
\begin{equation*}
    G(A,B)=A^{1/2} \left( A^{-1/2} B A^{-1/2} \right)^{1/2} A^{1/2} = G(B,A).
\end{equation*}
This definition was rigorously analyzed in \cite{ALM04}, where the authors also listed the main properties that a proper geometric mean should satisfy (see also \cite{M05,N09} for alternative definitions and generalizations to the case of geometric means of more than two matrices).

In the case of matrix-sequences, especially in the case of sequences in the GLT algebra, we may consider, and in fact it is quite crucial for the advancement in the present work, spectral properties of sequences of geometric means. In this regard, we have the following result from \cite{IKLS25}.

\begin{thm}
    Assume that $\{A_n\}\sim_{\text{GLT}} \kappa$ and $\{B_n\}\sim_{\text{GLT}} \xi $, where $A_n, B_n$ are Hermitian positive definite for every $n$. Then
    \begin{align*}
        \{G(A_n,B_n)\}\sim_{\text{GLT}} (\kappa \xi)^{1/2},\\
        \{G(A_n,B_n)\}\sim_{\sigma,\lambda} (\kappa \xi)^{1/2}.
    \end{align*}
\end{thm}

\section{Proof of the main result}\label{sec:main}

In this section, we prove our main result, Theorem \ref{main_thm}. In order to prove the theorem, we need some auxiliary results dealing with the approximation, in the Frobenius norm, of geometric means of Toeplitz matrices.

\begin{lem} \label{lemma_XY} 
    Let $ X $, $ Y $, $ \tilde{ X } $, $ \tilde{Y} $ be square matrices; then
    \begin{equation*}
        \norm{ X Y - \tilde{ X } \tilde{ Y } }_{ F } \leq \norm{ X }_{ S, \infty } \norm{ Y - \tilde{ Y } }_{ F } + \norm{ \tilde{ Y } }_{ S, \infty } \norm{ X - \tilde{ X } }_{ F }. 
    \end{equation*}
\end{lem}
\begin{proof}
    By a straightforward computation, we have
    \begin{align*}
        \norm{ X Y - \tilde{ X } \tilde{ Y } }_{ F } &= \norm{ X Y - X \tilde{ Y } + X \tilde{ Y } - \tilde{ X } \tilde{ Y } }_{ F } = \norm{ X \left( Y - \tilde{ {Y} } \right) + \left( X - \tilde{ X } \right) \tilde{ Y } }_{ F } \\
         & \leq \norm{ X }_{ S, \infty } \norm{ Y - \tilde{ Y } }_{ F } + \norm{ \tilde{ Y } }_{ S, \infty } \norm{ X - \tilde{ X } }_{ F },
    \end{align*}
    where the latter inequality follows from the discrete form of the generalized Hölder's inequality; the thesis is then proven.
\end{proof}

The following series of lemmas shows that the geometric means of uniformly bounded sequences is still uniformly bounded and gives useful estimates of differences of geometric means in the Frobenius norm, which will be used for our main result.

\begin{lem} \label{schatten_geom} 
    Let $ f, g \in L^{ \infty } \left( \left[ -\pi, \pi \right] \right) $ be real-valued almost everywhere functions such that $ \essinf f, \essinf g > 0 $. \\
    Then there exist  positive constants $ C_{1} $ and $ C_{2} $, both independent of $ n $, such that
    \begin{equation*}
        \norm{ G \left( T_{ n } \left( f \right), T_{ n } \left( g \right) \right) }_{ S, \infty } < C_1
    \end{equation*}
    and
    \begin{equation*}
        \norm{ G \left( T_{ n } \left( f \right)^{-1}, T_{ n } \left( g \right) \right) }_{ S, \infty } < C_2.
    \end{equation*}
\end{lem}
\begin{proof}
    By definition of geometric mean, we have 
    \begin{align*}
        \norm{ G \left( T_{ n } \left( f \right), T_{ n } \left( g \right) \right) }_{ S, \infty } &\leq \norm{ T_{ n } \left( f \right) }_{ S, \infty } \norm{ T_{ n } \left( f \right)^{-1/2} _{ n } T_n \left( g \right) T_n \left( f \right)^{-1/2} }_{ S, \infty }^{ 1 / 2 } \\
        & \leq \norm{ T_{ n } \left( f \right) }_{ S, \infty } \norm{ T_{ n } \left( f \right)^{-1} }_{ S, \infty }^{1/2} \norm{ T_{ n } \left( g \right) }_{ S, \infty }^{1/2} \\
        & \leq \frac{1}{\sqrt{\lambda_{\min} \left( T_n \left(f \right) \right)}} \norm{ T_{ n } \left( f \right) }_{ S, \infty } \norm{ T_{ n } \left( g \right) }_{ S, \infty }^{1/2}.
    \end{align*}
    By \cite[Theorem 6.2]{book-glt-1} the Schatten-$\infty$ norm of $ T_{ n } \left( f \right) $ is bounded by $ \norm{ f }_{ \infty } $, and the same applies to $ \norm{ T_{ n } \left( g \right) }_{ S, \infty } $. Moreover, \cite[Theorem 6.1]{book-glt-1} guaranties $\lambda_{\min} \left( T_n \left(f \right) \right) \geq \essinf f > 0 $, so that the first statement of the thesis follows. \\
    Again using the fact that the spectrum of $ T_{ n } \left( f \right) $ is contained in the essential range of $ f $, so that $ 0 < \essinf f < \lambda_{\min} \left( T_{ n} \left( f \right) \right) < \esssup f $, we get
    \begin{align*}
        \norm{ G \left( T_{ n } \left( f \right)^{ - 1 }, T_{ n } \left( g \right) \right) }_{ S, \infty } &\leq \norm{ T_{ n } \left( f \right)^{-1/2} }_{ S, \infty }^{ 2 } \norm{ T_{ n } \left( f \right)^{1/2} _{ n } T_n \left( g \right) T_n \left( f \right)^{1/2} }_{ S, \infty }^{ 1 / 2 } \\
        & \leq \frac{ 1 }{ \lambda_{\min} \left( T_n \left( f \right) \right) }\norm{ T_{ n } \left( f \right) }_{ S, \infty }^{ 1 / 2 } \norm{ T_{ n } \left( g \right) }_{ S, \infty }^{ 1 / 2 },
    \end{align*}
    leading to the thesis.
\end{proof}

\begin{defn} \label{def_approx_circ} 
    Let $ f \in L^{ \infty } \left( \left[ -\pi, \pi \right] \right) $ real valued almost everywhere; we denote by $ \left\lbrace C_{ n } \left( f \right) \right\rbrace_n $ the sequence of circulant matrices such that $ C_{ n } \left( f \right)= \arg\min_{X \in \mathcal{C}(n)} \norm{ X - T_{ n } \left( f \right) }_{F} $, where $\mathcal{C}(n)$ denotes the set of all circulant matrices of dimension $n$.    
    The existence of such an approximation is well-known in the literature as Chan preconditioner (see, e.g., \cite{WZ17} and references therein).
    Moreover, whenever $ T_{ n } \left( f \right) $ is Hermitian positive definite, then $ C_{ n } \left( f \right) $ is Hermitian positive definite.    
\end{defn}
\begin{lem} \label{lemma_2}
    Let $ f, g \in L^{ \infty } \left( \left[ -\pi, \pi \right] \right) $ be  real-valued a.e. functions such that $ \essinf f>0, \essinf g > 0 $; then
    \begin{equation*}
        \norm{ G \left( T_{ n } \left( f \right), T_{ n } \left( g \right) \right) - G \left( C_{ n } \left( f \right), C_{ n } \left( g \right) \right) }_{ F } = o \left( \sqrt{ n } \right)
    \end{equation*}
    and
    \begin{equation*}
        \norm{ G \left( T_{ n } \left( f \right)^{ -1 }, T_{ n } \left( g \right) \right) - G \left( C_{ n } \left( f \right)^{ -1 }, C_{ n } \left( g \right) \right) }_{ F } = o \left( \sqrt{ n } \right).
    \end{equation*}
\end{lem}
\begin{proof}
    For the sake of a lighter notation, let $ A = T_{ n } \left( f \right) $, $ B = T_{ n } \left( g \right) $ and $ \tilde{ A } = C_{ n } \left( f \right) $, $ \tilde{ B } = C_{ n } \left( g \right) $. \\
    We have
    \begin{align*}
        \norm{ G \left( A, B \right) - G \left( \tilde{ A }, \tilde{ B } \right) }_{ F } &= \norm{A^{1/2}\left( A^{-1/2}B A^{-1/2} \right)^{1/2} A^{1/2} - \tilde{A}^{1/2}\left( \tilde{A}^{-1/2} \tilde{B} \tilde{A}^{-1/2} \right)^{1/2} \tilde{A}^{1/2}}_F\\
        &\leq \norm{A^{1/2}-\tilde{A}^{1/2}}_F \norm{\left( A^{-1/2}B A^{-1/2} \right)^{1/2} A^{1/2}}_{S,\infty} \\
        &+ \norm{\tilde{A}^{1/2}}_{S,\infty} \norm{A^{1/2}}_{S,\infty} \norm{\left( A^{-1/2}B A^{-1/2} \right)^{1/2} - \left( \tilde{A}^{-1/2} \tilde{B} \tilde{A}^{-1/2} \right)^{1/2}}_{F} \\
        &+ \norm{A^{1/2}-\tilde{A}^{1/2}}_F \norm{\tilde{A}^{1/2}}_{S,\infty} \norm{\left( \tilde{A}^{-1/2} \tilde{B} \tilde{A}^{-1/2} \right)^{1/2}}_{S,\infty}.
    \end{align*}
Since all the terms in the Schatten-$\infty$ norm are uniformly bounded, we only need to estimate $\norm{A^{1/2}-\tilde{A}^{1/2}}_F$ and $\norm{\left( A^{-1/2}B A^{-1/2} \right)^{1/2} - \left( \tilde{A}^{-1/2} \tilde{B} \tilde{A}^{-1/2} \right)^{1/2}}_{F}$ to conclude. Since $\Lambda(A) \subset [\essinf f, \esssup f]$, by the Lipschitz continuity of $t \mapsto \sqrt{t}$ on sets bounded away from $0$, we have
\begin{equation*}
    \norm{A^{1/2}-\tilde{A}^{1/2}}_F \leq L \norm{A-\tilde{A}}_F = o(\sqrt{n}).
\end{equation*}
Moreover,
\begin{equation*}
    A^{-1/2}B A^{-1/2} \sim A^{-1/2} A^{-1/2} B A^{-1/2} A^{1/2} = A^{-1} B,
\end{equation*}
so that $\Lambda (A^{-1/2}BA^{-1/2}) \subset [\essinf(g/f),\esssup(g/f)]$ (see \cite{BS99}), and
\begin{equation*}
    \tilde{A}^{-1/2}\tilde{B} \tilde{A}^{-1/2} = \tilde{A}^{-1} \tilde{B},
\end{equation*}
so that $\Lambda (\tilde{A}^{-1/2}\tilde{B}\tilde{A}^{-1/2}) \subset [\essinf(g/f),\esssup(g/f)]$.
Using again the Lipschitz continuity of $t \mapsto \sqrt{t}$ on sets bounded away from $0$,
\begin{align*}
    \norm{\left( A^{-1/2}B A^{-1/2} \right)^{1/2} - \left( \tilde{A}^{-1/2} \tilde{B} \tilde{A}^{-1/2} \right)^{1/2}}_{F} &\leq L \norm{ A^{-1/2}B A^{-1/2} -  \tilde{A}^{-1/2} \tilde{B} \tilde{A}^{-1/2}}_{F} \\
    &\leq L \norm{A^{-1/2}-\tilde{A}^{-1/2}}_{F} \norm{BA^{-1/2}}_{S,\infty} \\
    &+ L\norm{\tilde{A}^{-1/2}}_{S,\infty} \norm{A^{-1/2} }_{S,\infty} \norm{B-\tilde{B}}_{F}\\
    &+ L \norm{A^{-1/2}-\tilde{A}^{-1/2}}_{F} \norm{A^{-1/2} B}_{S,\infty}.
\end{align*}
Noticing that $A^{-1/2}-\tilde{A}^{-1/2} = A^{-1/2}\left(\tilde{A}^{1/2}-A^{1/2}\right)\tilde{A}^{-1/2}$ and proceeding as above, we get
\begin{equation*}
    \norm{\left( A^{-1/2}B A^{-1/2} \right)^{1/2} - \left( \tilde{A}^{-1/2} \tilde{B} \tilde{A}^{-1/2} \right)^{1/2}}_{F} = o(\sqrt{n}),
\end{equation*}
as desired to conclude.

    The same technique with minimal variations can be applied to infer the second statement.
\end{proof}
\begin{cor} \label{cor_1} 
    Let $f_j,g_j$, $j=1,\dots,k$, be real-valued a.e. functions such that $\essinf f_j>0, \essinf g_j > 0$ for every $j$. Then
    \begin{equation*}
        \norm{ \prod_{ j = 1 }^{ k } G \left( T_{ n } \left( f_j \right), T_{ n } \left( g_j \right) \right) - \prod_{ j = 1 }^{ k } G \left( C_{ n } \left( f_j \right), C_{ n } \left( g_j \right) \right) }_{ F } = o \left( \sqrt{ n } \right),
    \end{equation*}
    and 
   \begin{equation*}
   \norm{ \prod_{ j = 1 }^{ k } G \left( T_{ n } \left( f_j \right)^{ -1 }, T_{ n } \left( g_j \right) \right) - \prod_{ j = 1 }^{ k } G \left( C_{ n } \left( f_j \right)^{ -1 }, C_{ n } \left( g_j \right) \right) }_{ F } = o \left( \sqrt{ n } \right).
   \end{equation*}
\end{cor}
\begin{proof}
    Let $ G_{ j } = G \left( T_{ n } \left( f_j \right), T_{ n } \left( g_j \right) \right) $ and $ \tilde{ G }_{ j } = G \left( C_{ n } \left( f_j \right), C_{ n }\left( g_j \right) \right) $; we proceed by induction on $k$. \\
    For $k=1$, the thesis reduces to the expression in Lemma \ref{lemma_2}. Let $k>1$. By Lemma \ref{lemma_XY}, we have 
    \begin{align*}
        \norm{\prod_{j=1}^{k} G_{ j } - \prod_{j=1}^{k} \tilde{ G }_{ j } }_{ F } & = \norm{ \left(\prod_{j=1}^{k-1} G_{ j } \right) G_{k} - \left(\prod_{j=1}^{k-1} \tilde{ G }_{ j }\right) \tilde{ G }_{ k } }_{ F } 
        \\
        &\leq \norm{ \prod_{j=1}^{k-1} G_{ j } }_{ S, \infty } \norm{ G_{ k } - \tilde{ G }_{ k } }_{ F } + \norm{ \tilde{ G }_{ k } }_{ S, \infty } \norm{ \prod_{j=1}^{k-1} G_{ j } - \prod_{j=1}^{k-1} \tilde{ G }_{ j } }_{ F }
    \end{align*}
    By Lemma \ref{schatten_geom}, Lemma \ref{lemma_2} and the sub-multiplicativity of the Schatten-$\infty$ norm, $ \norm{\prod_{j=1}^{k-1} G_{ j } }_{ S, \infty } \leq C $ for some $ C > 0 $ and $\norm{ G_k - \tilde{G}_k} = o (\sqrt{n})$; moreover, by induction, $ \norm{ \prod_{j=1}^{k-1} G_{ j } - \prod_{j=1}^{k-1} \tilde{ G }_{ j } }_{ F } = o \left( \sqrt{ n } \right) $ and $\norm{\tilde{G}_k}_{S, \infty}$ is uniformly bounded in $n$; hence, the first part of the thesis follows.  \\
    As above, the same procedure allows us to infer the second statement.
\end{proof}

\begin{lem} \label{lemma_1} 
    Let $ f, g \in L^{ \infty } \left( \left[ -\pi, \pi \right] \right) $ be real valued almost everywhere such that $ \essinf f>0, \essinf g > 0 $ ; then  
    \begin{equation*}
        \norm{ G \left( T_{ n } \left( f \right), T_{ n } \left( g \right)^{ -1 } \right) T_{ n } \left( g \right) - G \left( T_{ n } \left( f \right), T_{ n } \left( g \right) \right) }_{ F } = o \left( \sqrt{ n } \right)
    \end{equation*}
\end{lem}
\begin{proof}
    
    Let $X_n =T_{ n } \left( f \right)^{-1/2} T_n \left( g \right) T_n \left( f \right)^{-1/2}$, $Y_n = T_n \left( f \right)^{-1}X_n^{-1} T_n \left( f \right)^{-1}$ and $Z_n= T_n \left( f \right) X_n= T_n \left( f \right)^{1/2} T_n \left( g \right) T_n \left( f \right)^{-1/2}$ and let $ \tilde{ X }_n $, $ \tilde{ Y }_n $ and $ \tilde{ Z }_n $ be their corresponding circulant versions. \\
    Note that, by the definition of $X_n$, $T_n \left( g \right)= T_n \left( f \right)^{1/2} X_n T_n \left( f \right)^{1/2}$, so that $T_{ n } \left( g \right)^{-1}= T_{ n } \left( f \right)^{-1/2} X_n^{-1} T_n \left( f \right)^{-1/2}$. 
Since $\left\lbrace T_n \left( f \right) \right\rbrace_n$ and $ \left\lbrace T_n \left( g \right) \right\rbrace_n$ are uniformly bounded, so are $ \left\lbrace X_n \right\rbrace_n$, $\left\lbrace Y_n \right\rbrace_n$ and $\left\lbrace Z_n \right\rbrace_n$. \\
As in Lemma \ref{lemma_2}, let  $ A = T_{ n } \left( f \right) $, $ B = T_{ n } \left( g \right) $ and $ \tilde{ A } = C_{ n } \left( f \right) $, $ \tilde{ B } = C_{ n } \left( g \right) $.
Now, for notational convenience, we drop the index $ n $ for $ X_n $, $ Y_n $, $ Z_n $ and $ \tilde{ X }_n $, $ \tilde{ Y }_n $, $ \tilde{ Z }_n $ throughout the proof. \\    
Consider now the difference in which we are interested, i.e., $G(A, B^{-1} )B- G(A,B)$; we have
\begin{align*}
    G(A, B^{-1} )B- G(A,B)&=A^{1/2}\left( A^{-1/2}B^{-1}A^{-1/2}\right)^{1/2} A^{1/2} B - A^{1/2}\left( A^{-1/2} B A^{-1/2}\right)^{1/2}A^{1/2}\\
    &=A^{1/2}Y^{1/2}A^{1/2} B- A^{1/2} X^{1/2}A^{1/2}\\
    &=A^{1/2}\left( Y^{1/2} A^{1/2} B A^{-1/2} - X^{1/2} \right) A^{1/2}\\
    &= A^{1/2} \left( Y^{1/2} Z - X^{1/2} \right) A^{1/2}.
\end{align*}
From the previous equality, we immediately get
\begin{equation*}
    \norm{ G(A, B^{-1} )B- G(A,B) }_{ F } \leq \norm{ A^{1/2} }_{ S, \infty }^{2} \norm{ Y^{1/2}Z-X^{1/2} }_{ F }.
\end{equation*}
Since $\norm{A^{1/2}}_{S,\infty} $ is (uniformly) bounded, we only need to prove that
\begin{equation*}
    \norm{ Y^{1/2}Z-X^{1/2} }_F =o \left( \sqrt{n} \right)
\end{equation*}
as $n\to \infty$ to conclude.
Recalling the commutativity of circulant matrices, a straightforward computation yields $\tilde A^{1/2} \left( \tilde Y^{1/2} \tilde Z - \tilde X^{1/2} \right) \tilde A^{1/2} =0$, so that
\begin{align*} 
    \norm{ Y^{1/2}Z-X^{1/2} }_F &= \norm{Y^{1/2}Z-X^{1/2} - \tilde Y^{1/2} \tilde Z+ \tilde X^{1/2} }_F \\
    &\leq \norm{ \tilde X^{1/2} - X^{1/2} }_F + \norm{ \tilde Y^{1/2} \tilde Z - Y^{1/2} Z}_F.
\end{align*}
Let $E_1= \tilde X^{1/2} -  X^{1/2}$ and  $E_2=\tilde Y^{1/2} \tilde Z - Y^{1/2}  Z$ and estimate $\norm{E_1}_{F} $ and $\norm{E_2}_{F}$ to conclude. \\ 
We begin by providing a bound for $ \norm{ E_1 }_{ F } $. First of all, since $t \mapsto \sqrt{t}$ is Lipschitz continuous and noticing that, just as in Lemma \ref{lemma_2}, $X=A^{-1/2}BA^{-1/2}$ and $\tilde{X}=\tilde{A}^{-1/2}\tilde{B}\tilde{A}^{-1/2}$  have spectrum in an interval uniformly bounded away from $0$, we get $\norm{ \tilde X^{1/2}-X^{1/2} }_F \leq L \norm{ \tilde X - X }_F$. Now notice that
\begin{align*}
    \tilde X-X &= \tilde A^{-1/2} \tilde B \tilde A^{-1/2} - A^{-1/2} B A^{-1/2} \\
               &=\left( \tilde A^{-1/2} - A^{-1/2} \right) B A^{-1/2} + \tilde A^{-1/2} \left( \tilde B - B \right) A^{-1/2} + \tilde A^{-1/2} \tilde B \left( \tilde A^{-1/2} - A^{-1/2}\right),  
\end{align*}
so that
\begin{align*}
    \norm{ \tilde X^{1/2}  - X^{1/2} }_F & \leq L \norm{ \tilde X  - X }_F \\ 
    & \leq L \norm{ \tilde A^{-1/2}-A^{-1/2} }_F \left( \norm{ A^{-1/2} }_{S,\infty} \norm{ B }_{ S, \infty } + \norm{ \tilde A^{-1/2} }_{S,\infty} \norm{ \tilde B }_{S,\infty} \right) \\
    & + L \norm{ \tilde B -B }_F \norm{ A^{-1/2} }_{S,\infty} \norm{ \tilde A^{-1/2} }_{S,\infty}.
\end{align*}
Since all the terms in Schatten-$\infty$ norm are bounded by a constant, $\norm{ \tilde A^{-1/2}- A^{-1/2} }_F \leq L \norm{ \tilde A-A }_F = o\left(\sqrt{n} \right)$, and $\norm{ \tilde B - B }_F=o\left(\sqrt{n} \right)$, we conclude that $\norm{E_1}_F=o\left(\sqrt{n}\right)$. \\
A similar technique can be exploited for bounding $ \norm{ E_2 }_{ F } $; by Lemma \ref{lemma_XY} 
\begin{equation*}
    \norm{ E_2 }_F \leq \norm{ \tilde Y^{1/2} }_{S, \infty} \norm{ \tilde Z - Z }_F + \norm{ Z }_{S, \infty} \norm{ \tilde Y^{1/2} - Y^{1/2} }_F.
\end{equation*}
Notice that the Shatten-$\infty$ norms are bounded and, again by Lemma \ref{lemma_XY},
\begin{align*}
    \norm{ \tilde Z - Z }_F = \norm{ \tilde A \tilde X - AX }_F  \leq \norm{ \tilde A }_{S, \infty} \norm{ \tilde X - X }_F + \norm{ X }_{S, \infty} \norm{ \tilde A - A }_F = o \left( \sqrt{n} \right).
\end{align*}
Moreover, note that $Y=A^{-1}X^{-1}A^{-1}=A^{-1/2}B^{-1}A^{-1/2} \sim B^{-1}A^{-1}$ and $\tilde{Y} =\tilde{A}^{-1/2} \tilde{B}^{-1} \tilde{A}^{-1/2}=\tilde{B}^{-1} \tilde{A}$, so that applying the same reasoning as for $X$ and $\tilde{X}$ and using the Lipschitz continuity of $t \mapsto \sqrt{t}$, we get
\begin{align*}
    \norm{ \tilde Y^{1/2} - Y^{1/2} }_F & \leq L \norm{ \tilde Y - Y }_F = \norm{ \tilde A^{-1} \tilde X^{-1} \tilde A^{-1} -  A^{-1}  X^{-1}  A^{-1} }_F \\
    &= \norm{ \tilde A^{-1} \tilde X^{-1} \tilde A^{-1} - \tilde A^{-1} \tilde X^{-1}  A^{-1} + \tilde A^{-1} \tilde X^{-1}  A^{-1} -  A^{-1} X^{-1}  A^{-1} }_F \\
    & \leq \norm{ F_1 }_{ F } + \norm{ F_2 }_{ F },
\end{align*}
where $ F_1 = \tilde A^{-1} \tilde X^{-1} \tilde A^{-1} - \tilde A^{-1} \tilde X^{-1}  A^{-1} $ and $ F_2 = \tilde A^{-1} \tilde X^{-1}  A^{-1} -  A^{-1} X^{-1}  A^{-1} $. \\ 
Clearly,
\begin{align*}
    \norm{ F_1 }_{ F } &\leq \norm{ \tilde A^{-1} }_{S, \infty} \norm{\tilde X^{-1} \tilde A^{-1} - \tilde X^{-1} A^{-1} }_F \\
    &\leq \norm{ \tilde A^{-1} }_{S, \infty}  \norm{\tilde X^{-1} }_{S, \infty} \norm{\tilde A^{-1} - A^{-1} }_{F} = o \left(\sqrt{n} \right),  
\end{align*}
where we used that $\norm{ \tilde A^{-1} - A^{-1} }_F = \norm{ \tilde A^{-1}(A-\tilde A)A^{-1}}_F \leq C \norm{ \tilde A - A }_F=o\left(\sqrt{n}\right)$, and
\begin{align*}
    \norm{ F_2 }_{ F } &= \norm{ \left(  \tilde A^{-1} \tilde X^{-1}  - A^{-1}  X^{-1} \right)A^{-1} }_F \\
                       &\leq \norm{ \tilde A^{-1} }_{S, \infty} \norm{ \tilde X^{-1} - X^{-1} }_F \norm{A^{-1} }_{S, \infty} + \norm{ A^{-1} }_{S, \infty} \norm{ X^{-1} }_{S, \infty} \norm {\tilde A^{-1} - A^{-1} }_F = o \left(\sqrt{n} \right),
\end{align*}
by Lemma \ref{lemma_XY}. \\ 
Hence $\norm{E_2}_F = o \left(\sqrt{n}\right)$ and we conclude
\begin{equation*} 
    \norm{ G\left(A, B^{-1} \right)B- G \left(A,B \right) }_F = o \left( \sqrt{n} \right).
\end{equation*}
\end{proof}

Let $\mathcal{A}_n$ be as in \eqref{eq:A_Toeplitz}. For notational convenience, we let, for $j=1,\dots, k$, $ A_{ j } = T_{ n } \left( f_{ j, j } \right) $ and, for $j=1,\dots,k-1$, $ B_{ j } = T_{ n } \left( f_{ j +1, j } \right) $ and $ C_{ j } = T_{ n } \left( f_{ j, j + 1 } \right) $. \\ 
Consider now the block-diagonal matrix $ \mathcal{E}_n $ given by
\begin{equation*}
    \mathcal{E}_n = 
   \left[ 
   \begin{array}{ccc}
        E_{ 1 } & &  \\
        & \ddots & \\
        & & E_{ k }
   \end{array}
   \right], \text{ with }
   E_{ j } = 
   \begin{cases}
       I_{ n } \text{ for } j = 1, \\
       E_{ j - 1 } G \left( B_{ j -1 }^{ -1 }, C_{ j -1 } \right) \text{ for } j = 2, \dots, k,
   \end{cases} 
\end{equation*}
so that both $ \mathcal{ A }_n $ and $ \mathcal{ E }_n $ are square matrices of size $ kn $. 
\begin{thm}[] \label{thm_1} 
    With the previously introduced notation,
    \begin{equation*}
        \norm{ \mathcal{E}_n \mathcal{A}_n \mathcal{E}_n^{ - 1 } - \hat{\mathcal{A }}_n }_{ F } = o \left( \sqrt{n} \right)
    \end{equation*}
    as $ n\to \infty $, where
    \begin{equation*}
        \hat{\mathcal{A}}_n = 
        \left[ 
        \begin{array}{ccccc}
            A_{1} & G_1 &  &  &  \\
            G_1 & A_{2} & G_2 &  &  \\
             & G_2 & A_{3} &  &  \\
             &  &  &  & G_{k-1} \\
             &  &  & G_{k-1} & A_{k}
        \end{array}
        \right], \text{ with } G_{ j } = G \left( B_{ j }, C_{ j } \right) \text{ for } j=1,\dots,k-1.
    \end{equation*}
\end{thm}
\begin{proof}
    We prove that each block of $\mathcal{E}_n \mathcal{A}_n \mathcal{E}_n^{ - 1 } - \hat{\mathcal{A } }_n$ has Frobenius norm of order $o \left( \sqrt{n} \right)$; since the total number of non-zero blocks is $3k-2$ independent of $ n $, the claim follows. \\
We begin by estimating the Frobenius norm of the generic $j$-th lower diagonal block, for $j=1,\dots,k-1$, explicitly given by $E_{j+1} B_j E_{j}^{-1} - G_{ j }$:
\begin{align*}
    \norm{E_{j+1} B_j E_{j}^{-1} - G_{ j }}_F &= \norm{E_{j+1} B_j E_{j}^{-1} - E_{j} G_j E_{j}^{-1} + E_{j} G_j E_{ j }^{-1} - G_{ j } }_F \\
    & \leq  \norm{E_{j+1} B_j E_{j}^{-1} - E_{j} G_j E_{ j }^{-1} }_F + \norm{E_{j} G_j E_{ j }^{-1} - G_{ j }}_F .
\end{align*}
We proceed by investigating each term independently: 
\begin{align*}
    \norm{E_{j+1} B_j E_{j}^{-1} - E_{j} G_j E_{ j}^{-1} }_F &= \norm{E_{j} \left( G \left( B_{j}^{-1}, C_{j} \right) B_{j} - G_{ j } \right) E_{j}^{-1}}_F \\
    &\leq \norm{E_{j}}_{S, \infty} \norm{ E_{j}^{-1}}_{S, \infty} \norm{G \left( B_{j}^{-1}, C_{j} \right) B_j - G_j}_F .
\end{align*}
By the sub-multiplicativity of Schatten $p$-norms and Lemma \ref{schatten_geom}, we have $\norm{E_{j}}_{S, \infty} \leq \prod_{i=1}^{j-1} \norm{G\left(B_{i}^{-1}, C_{i} \right)}_{S, \infty} \leq C$. \\
Moreover, we already proved in Lemma \ref{lemma_1} that $\norm{G \left( B_{j}^{-1}, C_{j} \right) B_j - G_j}_F = \norm{G \left( B_{j}^{-1}, C_{j} \right) B_j - G\left( B_j,C_j \right)}_F = o \left( \sqrt{n} \right)$; hence we conclude
\begin{equation*}
    \norm{E_{j+1} B_j E_{j}^{-1} - E_{j} G_j E_{j} }_F = o \left(\sqrt{n}\right).  
\end{equation*}
As for the second term, let $\tilde{E}_1 = I_n$ and $\tilde{E}_j = \tilde{E}_{j-1} G \left( \tilde{B}_{j-1}^{-1}, \tilde{C}_{j-1}  \right)$ for $j=2,\dots,k$; by exploiting the commutativity of the circulant algebra, we have $\tilde{E}_{j} G \left( \tilde{B}_j, \tilde{C}_j \right) \tilde{E}_{j}^{-1} = G \left( \tilde{B}_j, \tilde{C}_j \right)$; therefore, 
\begin{align*}
    \norm{E_{j} G_j E_{ j}^{-1} - G_{ j }}_F &= \norm{E_{j} G_j E_{ j }^{-1} - \tilde{E}_{j} G \left( \tilde{B}_j, \tilde{C}_j \right) \tilde{E}_{j}^{-1} + G \left( \tilde{B}_j, \tilde{C}_j \right) - G_j}_F \\
    & \leq \norm{E_{j} G_j E_{ j }^{-1} - \tilde{E}_{j} G \left( \tilde{B}_j, \tilde{C}_j \right) \tilde{E}_{j}^{-1}  }_F + \norm{ G \left( \tilde{B}_j, \tilde{C}_j \right) - G_j}_F.
\end{align*}
The term $\norm{ G \left( \tilde{B}_j, \tilde{C}_j \right) - G_j}_F = \norm{ G \left( \tilde{B}_j, \tilde{C}_j \right) - G \left( B_j, C_j \right)}_F = o\left( \sqrt{n} \right)$ by Lemma \ref{lemma_2}; for the first term we proceed as follows: $   $ 
\begin{align*}
    & \norm{E_{j} G_j E_{j}^{-1} - \tilde{E}_{j} G \left( \tilde{B}_j, \tilde{C}_j \right) \tilde{E}_{j}^{-1}  }_F\\
    &=  \Big\| E_{j} G_j E_{ j }^{-1} - \tilde{E}_{j} G_j E_{j}^{-1} + \tilde{E}_{j} G_j E_{j}^{-1} - \tilde{E}_{j} G \left( \tilde{B}_j, \tilde{C}_j \right) E_{j}^{-1} + \\
    &+\tilde{E}_{j} G \left( \tilde{B}_j, \tilde{C}_j \right) E_{j}^{-1} - \tilde{E}_{j} G \left( \tilde{B}_j, \tilde{C}_j \right) \tilde{E}_{j}^{-1}  \Big\|_F \\
    & \leq \norm{E_{j} G_j E_{j}^{-1} - \tilde{E}_{j} G_j E_{j}^{-1} }_F + \norm{ \tilde{E}_{j} G_j E_{j}^{-1} - \tilde{E}_{j} G \left( \tilde{B}_j, \tilde{C}_j \right) E_{j}^{-1}}_F \\
    & + \norm{\tilde{E}_{j} G \left( \tilde{B}_j, \tilde{C}_j \right) E_{j}^{-1} - \tilde{E}_{j} G \left( \tilde{B}_j, \tilde{C}_j \right) \tilde{E}_{j}^{-1}}_F \\
    & = \norm{\left(E_{j} - \tilde{E}_{j} \right)  G_j E_{j} }_F + \norm{\tilde{E}_{j} \left(G_j  - G \left( \tilde{B}_j, \tilde{C}_j \right)\right) E_{j}^{-1}}_F + \\
    & +\norm{\tilde{E}_{j} G \left( \tilde{B}_j, \tilde{C}_j \right) \left( E_{j}^{-1} - \tilde{E}_{j}^{-1} \right)}_F .
\end{align*}
The first term can be bounded by $\norm{G_j}_{S, \infty } \norm{E_{j}}_{S, \infty } \norm{E_{j} - \tilde{E}_{j}}_F $; by combining now Lemma \ref{schatten_geom} and Corollary \ref{cor_1}, we deduce $\norm{\left(E_{j} - \tilde{E}_{j} \right)  G_j E_{j} }_F = o \left(\sqrt{n} \right)$. \\
The same procedure can be applied with minimal variations to the terms $\norm{\tilde{E}_{j} \left(G_j  - G \left( \tilde{B}_j, \tilde{C}_j \right)\right) E_{j}^{-1}}_F$ and $\norm{\tilde{E}_{j} G \left( \tilde{B}_j, \tilde{C}_j \right) \left( E_{j} - \tilde{E}_{j} \right)}_F$, leading to the desired estimate, that is,
\begin{equation*}
    \norm{E_{j} G_j E_{ j }^{-1} - \tilde{E}_{j} G \left( \tilde{B}_j, \tilde{C}_j \right) \tilde{E}_{j}^{-1}  }_F = o \left( \sqrt{n} \right),
\end{equation*}
so that $ \norm{E_{j} G_j E_{ j }^{-1} - G_{ j }}_F = o \left( \sqrt{n} \right) $, allowing us to conclude $\norm{E_{j+1} B_j E_{j}^{-1} - G_{ j }}_F = o \left( \sqrt{n} \right)$. \\
For the upper diagonal block $E_{j} C_j E_{ j+1 }^{-1 } - G_{ j }$, $j=1,\dots,k-1$, the very same procedure leads to an analogous equality, that is, $\norm{E_{j} C_j E_{ j+1 }^{-1 } - G_{ j }}_F = o \left(\sqrt{n} \right)$. \\
Finally, we tackle the diagonal element: as above, we claim $ \norm{ E_{ j } A_j E_{j}^{-1} - A_{j} }_F = o \left(\sqrt{n} \right) $. \\
First of all, exploiting the commutativity of the circulant algebra, we rewrite the diagonal term as $ E_{ j } A_j E_{j}^{-1} - \tilde{E}_{j} \tilde{A}_{j} \tilde{E}_{j}^{-1} + \tilde{A}_{j} - A_{j} $, so that
\begin{equation*}
    \norm{ E_{ j } A_j E_{j}^{-1} - A_{j} }_F \leq \norm{ E_{ j } A_j E_{j}^{-1} - \tilde{E}_{j} \tilde{A}_{j} \tilde{E}_{j}^{-1} }_{ F } + \norm{ \tilde{A}_{j} - A_{j} }_{ F }.
\end{equation*}
Since $ \norm{ \tilde{A}_{j} - A_{j} }_{ F } = o \left( \sqrt{n} \right) $ by Definition \ref{def_approx_circ}, it is only left to prove $ \norm{ E_{ j } A_j E_{j}^{-1} - \tilde{E}_{j} \tilde{A}_{j} \tilde{E}_{j}^{-1} }_{ F } = o \left( \sqrt{n} \right) $. \\ 
Indeed,
\begin{align*}
    \norm{ E_{ j } A_j E_{j}^{-1} - A_{j} }_F & = \norm{ E_{ j } A_j E_{j}^{-1} - \tilde{E}_{ j} A_{ j } E_{ j }^{-1} +\tilde{E}_{ j} A_{ j } E_{ j }^{-1} - \tilde{E}_{j} \tilde{A}_{j} E_{j}^{-1} + \tilde{E}_{j} \tilde{A}_{j} E_{j}^{-1}- \tilde{E}_{j} \tilde{A}_{j} \tilde{E}_{j}^{-1} }_{ F } \\
     & \leq \norm{ \left( E_{j} - \tilde{E}_{j} \right) A_{j} E_{j}^{-1} }_{ F } + \norm{ \tilde{E}_{j} \left( A_{j} - \tilde{A}_{j} \right) E_{j}^{-1} }_{ F } + \norm{ \tilde{E}_{j} \tilde{A}_{j} \left( E_{j}^{-1} - \tilde{E}_{j}^{-1} \right)  }_{ F } \\
    & \leq C_{1} \norm{ E_{j} - \tilde{E}_{j} }_{ F } + C_{2} \norm{ A_{j} - \tilde{A}_{j} }_{ F } + C_{3} \norm{ E_{j}^{-1} - \tilde{E}_{j}^{-1} }_{ F },
\end{align*}
with $ C_{1} = \norm{ A_{j} E_{j}^{-1} }_{ S, \infty} $, $C_{2} = \norm{ \tilde{E}_{j} }_{ S, \infty} \norm{ E_{j}^{-1} }_{ S, \infty}$ and $ C_{3} = \norm{ \tilde{E}_{j} \tilde{A}_{j} }_{ S, \infty  } $, where we have already determined that all these quantities are bounded independently of $ n $. \\
By  Corollary \ref{cor_1}, we know that $ \norm{ E_{j} - \tilde{E}_{j} }_{ F } = o \left( \sqrt{n} \right) $ and, by Definition \ref{def_approx_circ}, $\norm{ A_{j} - \tilde{A}_{j} }_{ F } = o \left( \sqrt{n} \right) $. \\
Moreover, $ \norm{ E_{j}^{-1} - \tilde{E}_{j}^{-1} }_{ F } = \norm{ \tilde{E}_{j}^{-1} \left( E_{j} - \tilde{E}_{j} \right) E_{j}^{-1} }_{ F } \leq C_{4} \norm{ E_{j} - \tilde{E}_{j} }_{ F }$, with $ C_{4} = \norm{ E_{j}^{-1} }_{ S, \infty } \norm{ \tilde{E}_{j}^{-1} }_{ S, \infty } $, so that $ \norm{ E_{j} - \tilde{E}_{j} }_{ F } = o \left( \sqrt{n} \right) $. \\
Hence, we infer $ \norm{ E_{ j } A_j E_{j}^{-1} - A_{j} }_F = o \left(\sqrt{n} \right) $, concluding the proof. 
\end{proof}
We are now ready to prove our main result.
\begin{thm}\label{thm_final}
    Let $\mathcal{A}_n$ be as in Equation \eqref{eq:A_Toeplitz}. Assume that $f_{i,j} \in L^{\infty}([-\pi,\pi])$ and real-valued almost everywhere for every $i,j$. Moreover, assume that $\essinf f_{i, i+1}, \essinf f_{i+1, i} > 0$ for every $i=1,\dots,k-1$. Then
    \begin{equation*}
        \{\mathcal{A}_n\}_n \sim_{\lambda} F(\theta),
    \end{equation*}
    with
    \begin{equation*}
        F(\theta)= \begin{bmatrix}
f_{1,1} & f_{1,2} &  &  \\
f_{2,1} & f_{2,2} & \ddots &  \\
 & \ddots & \ddots & f_{k-1 , k} \\
 &  & f_{k, k-1} & f_{k, k}
\end{bmatrix}.
    \end{equation*}
    Equivalently,
    \begin{equation*}
        \{\mathcal{A}_n\}_n \sim_{\lambda} \tilde F(\theta),
    \end{equation*}
    with
    \begin{equation*}
        \tilde F(\theta) = \begin{bmatrix}
        f_{1,1} & f_{1,2}^{1/2}f_{2,1}^{1/2} &  &  \\
f_{1,2}^{1/2}f_{2,1}^{1/2} & f_{2,2} & \ddots &  \\
 & \ddots & \ddots & f_{k-1 , k}^{1/2}f_{k, k-1}^{1/2} \\
 &  & f_{k , k-1}^{1/2}f_{k-1 , k}^{1/2} & f_{k, k}
\end{bmatrix}.
    \end{equation*}
\end{thm}
\begin{proof}
    With the notation introduced in Theorem \ref{thm_1}, we have $ \left\lbrace \Pi_{ n } \hat{ \mathcal{ A } }_{n} \Pi_{ n }^{ T } \right\rbrace_n \sim_{\text{GLT}} \tilde{F} $, where $ \Pi_{ n } $ is a permutation matrix that depends only on $ n $ and $ k $ (see \cite[Section 2]{GMS18}).
    Since $ \Pi_n \hat{ A }_{n} \Pi_n^T $ is Hermitian for every $n$, we deduce $ \left\lbrace \Pi_{ n }\hat{ A }_{n} \Pi_{ n }^{T} \right\rbrace_n \sim_{\lambda} \tilde{F} $. \\
    Consider now
        \begin{equation*}
            \Pi_{n} \mathcal{ E }_n \mathcal{ A }_n \mathcal{ E }_{n}^{-1} \Pi_n^{T} = \Pi_{ n } \hat{ \mathcal{ A } }_{n} \Pi_{ n }^{ T } - \Pi_{ n } \left( \mathcal{ E }_n \mathcal{ A }_n \mathcal{ E }_{n}^{-1} - \hat{ \mathcal{ A }}_n \right) \Pi_{n}^{T}.
        \end{equation*}
        As a consequence of Theorem \ref{thm_1},
        \begin{equation*}
            \norm{ \Pi_{ n } \left( \mathcal{ E }_n \mathcal{ A }_n \mathcal{ E }_{n}^{-1} - \hat{ \mathcal{ A }}_n \right) \Pi_{n}^{T} }_{ F } = o \left( \sqrt{n} \right).
        \end{equation*}
        Hence, by \cite[Theorem 1]{BS20}, we deduce $ \left\lbrace \Pi_{n} \mathcal{ E }_n \mathcal{ A }_n \mathcal{ E }_{n}^{-1} \Pi_n^{T} \right\rbrace_n \sim_{\lambda } \tilde{F}$.
         Since the matrices $ \mathcal{ A }_n $ and $ \Pi_{n} \mathcal{ E }_n \mathcal{ A }_n \mathcal{ E }_{n}^{-1} \Pi_n^{T} $ are similar, their spectra coincide: therefore, $ \left\lbrace \mathcal{ A }_{n} \right\rbrace_n \sim_{\lambda} \tilde{ F } $. \\
    Finally, since $ F \left( \theta \right) $ is similar to $ \tilde{ F } \left( \theta \right) $, the conclusion follows.\\
\end{proof}

\section{Numerical tests}\label{sec:test}

{In this section, we perform numerical tests confirming the validity of the theory developed and we exploit it to propose a preconditioning strategy for the problem at hand. Therefore, we divide this section in two parts: the first one establishing the goodness of the spectral analysis and the second giving a proposal and numerical analysis of a preconditioner.}

\subsection{Theory validation through spectral analysis}\label{ssec:test_distr}

In this subsection, we numerically confirm the results in Theorem \ref{thm_final} in a number of situations. We divide our numerical tests into different groups based on the number $k \geq 2$ and on the properties of the generating functions of each block.

\subsubsection{Case $k=2$}

We consider the case where $k=2$ and make the following choices for $f_{i,j}$, $i,j=1,2$: $f_{1,1}(\theta)= 3-4\cos(\theta), f_{1,2}(\theta) = 4-\cos(\theta)-2\cos(2\theta), f_{2,1}(\theta)= \theta^2+1, f_{2,2}(\theta)=\theta-5 $. Note that in this case, we have satisfied the assumption $\essinf f_{1,2},\essinf f_{2,1} > 0$. In Figure \ref{Caso_1_facile} we empirically show the validity of Theorem \ref{thm_final}. Notice that here we only consider the real part of the eigenvalues of $\mathcal{A}_n$, since the complex part is numerically negligible, being of order $10^{-14}$.
\begin{figure}[htpb]
\centering
    \includegraphics[width=60mm]{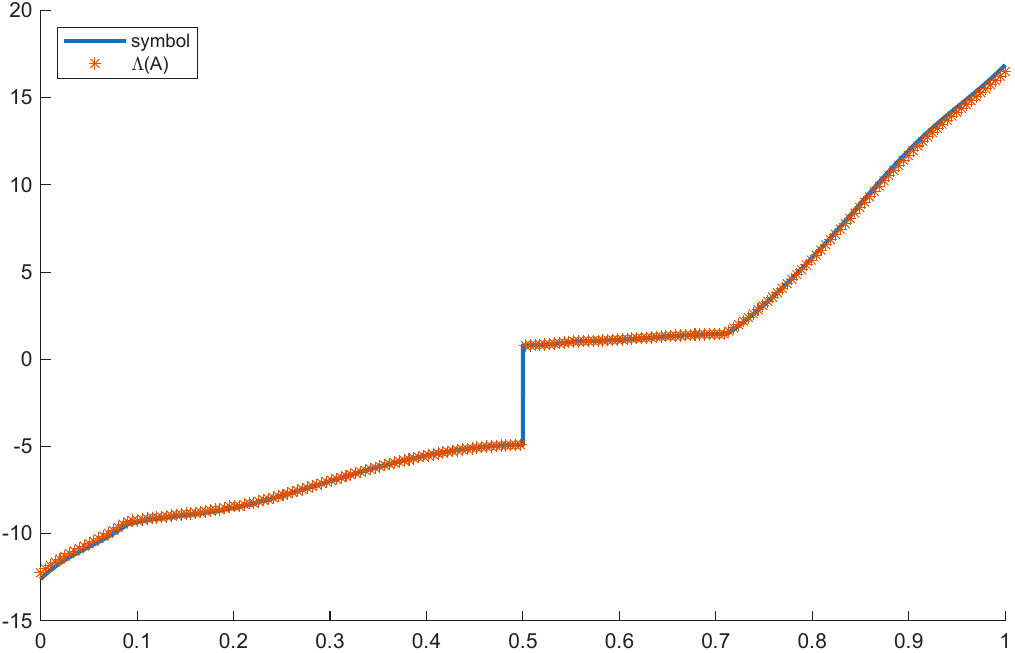}%
   \qquad \qquad
    \includegraphics[width=70mm]{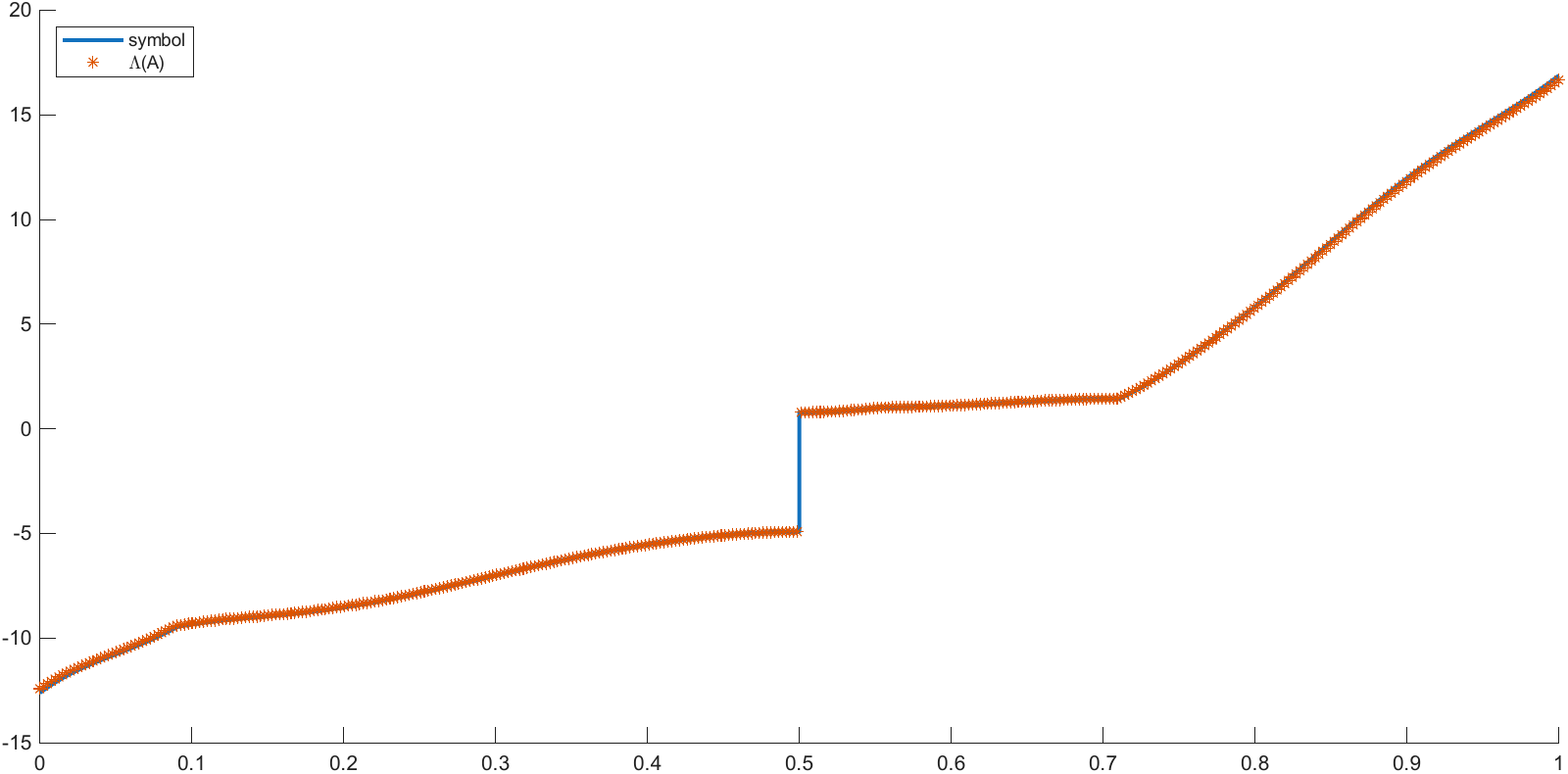}
    \caption{Comparison between the real part of the spectrum of $\mathcal{A}_n$ and a uniform sampling of $F(\theta)$ for $n=100,200$}
    \label{Caso_1_facile}
\end{figure}
\subsection{Case $k=2$ {(stretching the hypothesis)}}\label{ssec_k2_limit}

In this subsection, we let $k=2$ and analyze the case in which both $\essinf f_{1,2} $ and $\essinf f_{2,1}$ are equal to $0$, therefore breaking the hypothesis of our main theorem. This is motivated by the fact that, in order to study the distribution of geometric means of GLT sequences, no requirement is necessary, apart from the positive definiteness of the matrices involved (\cite{IKLS25}). The numerical experiments suggest that the technical hypothesis $\essinf f_{i,j} > 0$ might be abandoned (see Conjecture \ref{conj_final}).\\
We consider $f_{1,1}(\theta)= -\theta^4, f_{1,2}(\theta) = 2-2\cos(\theta), f_{2,1}(\theta)= \theta^2, f_{2,2}(\theta)=1 $. Note that, in this case, it still holds $\mathcal{L}\left(\{\theta \in [-\pi,\pi] \, : \,f_{i,j}(\theta)=0\}\right)=0$ for every $i,j=1,2$. Figure \ref{Caso_1_limite} shows that the same conclusion as in Theorem \ref{thm_final} seems to hold even under weaker assumptions on the generating functions of each block. As above, the complex part of the eigenvalues of $\mathcal{A}_n$ turns out to be numerically negligible, being again of order $10^{-14}$.

\begin{figure}[h]
\centering
    \includegraphics[width=70mm]{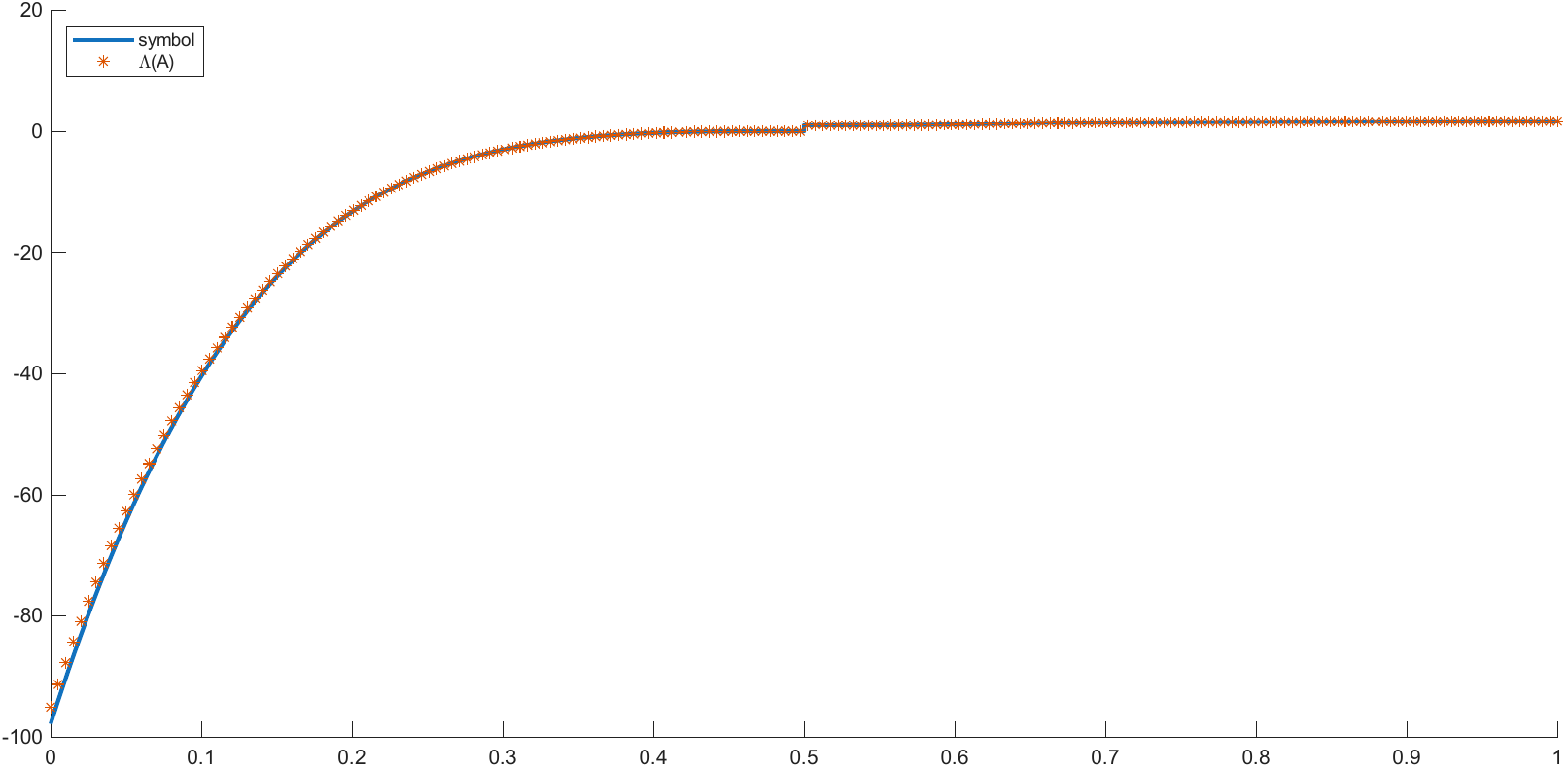}%
    \qquad \qquad
    \includegraphics[width=70mm]{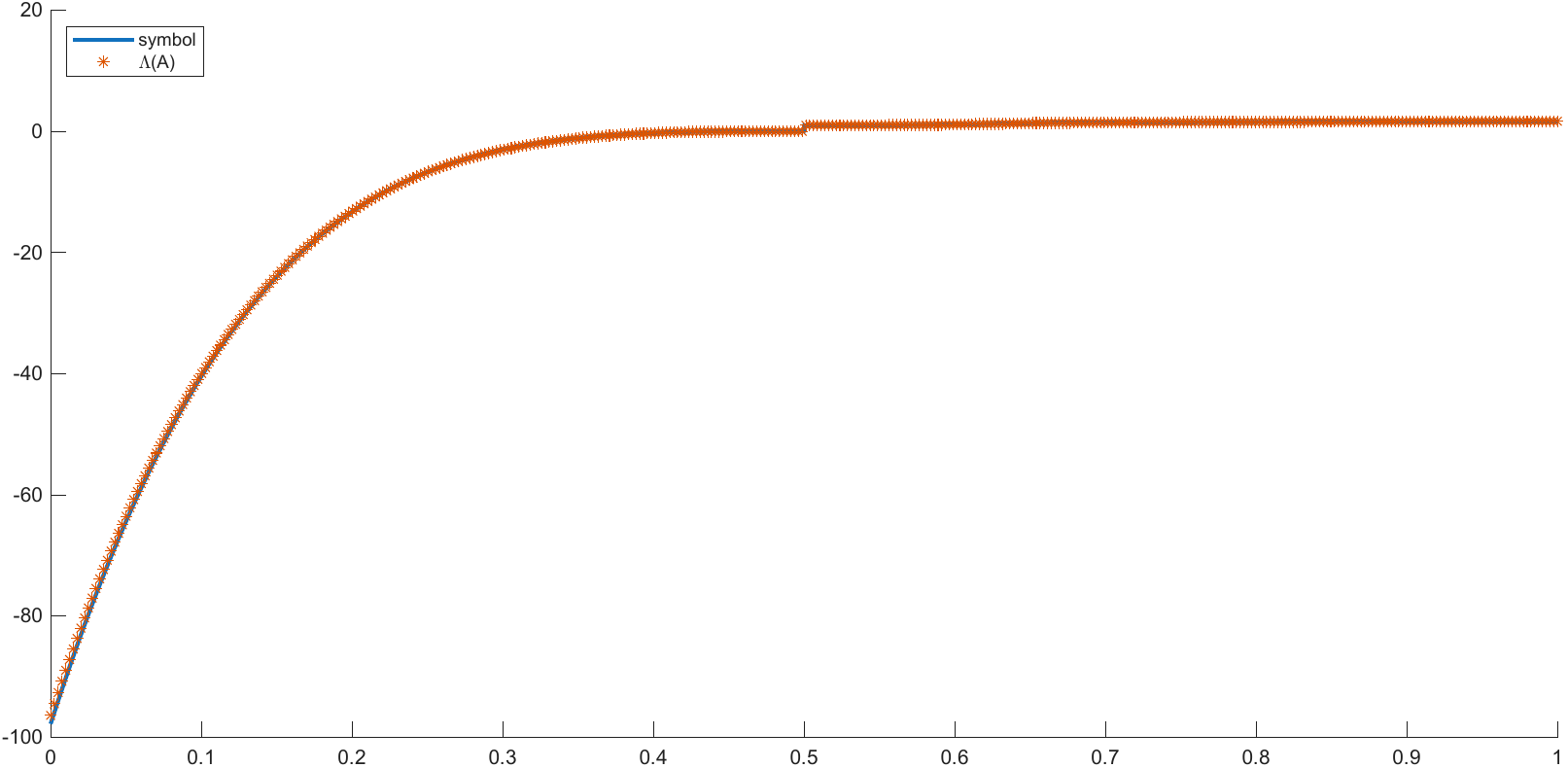}
    \caption{Comparison between the real part of the spectrum of $\mathcal{A}_n$ and a uniform sampling of $F(\theta)$ for $n=100,200$}
    \label{Caso_1_limite}
\end{figure}

\subsubsection{Case $k=3$ (the limit case)}

In this subsection, in the case $k=3$ and with the same theoretical setting as in Subsection \ref{ssec_k2_limit}, we analyze the spectral behavior of $\{\mathcal{A}_n\}_n$, with the following choices for $f_{i,j}$: $f_{1,1}(\theta)= \theta^2, f_{1,2}(\theta) = 2-2\cos(\theta), f_{2,1}(\theta)= \theta^2, f_{2,2}(\theta)=1, f_{2,3}(\theta)=\theta+4, f_{3,2}(\theta)=|\theta|, f_{3,3}(\theta)=\sqrt{|\theta|} $. Figure \ref{Caso_2_limite} confirms the agreement between the symbol function and the spectrum of the sequence. As above, the error in considering only the real part of the eigenvalues is numerically negligible, being of order $10^{-14}$.

\begin{figure}[H]
\centering
    \includegraphics[width=70mm]{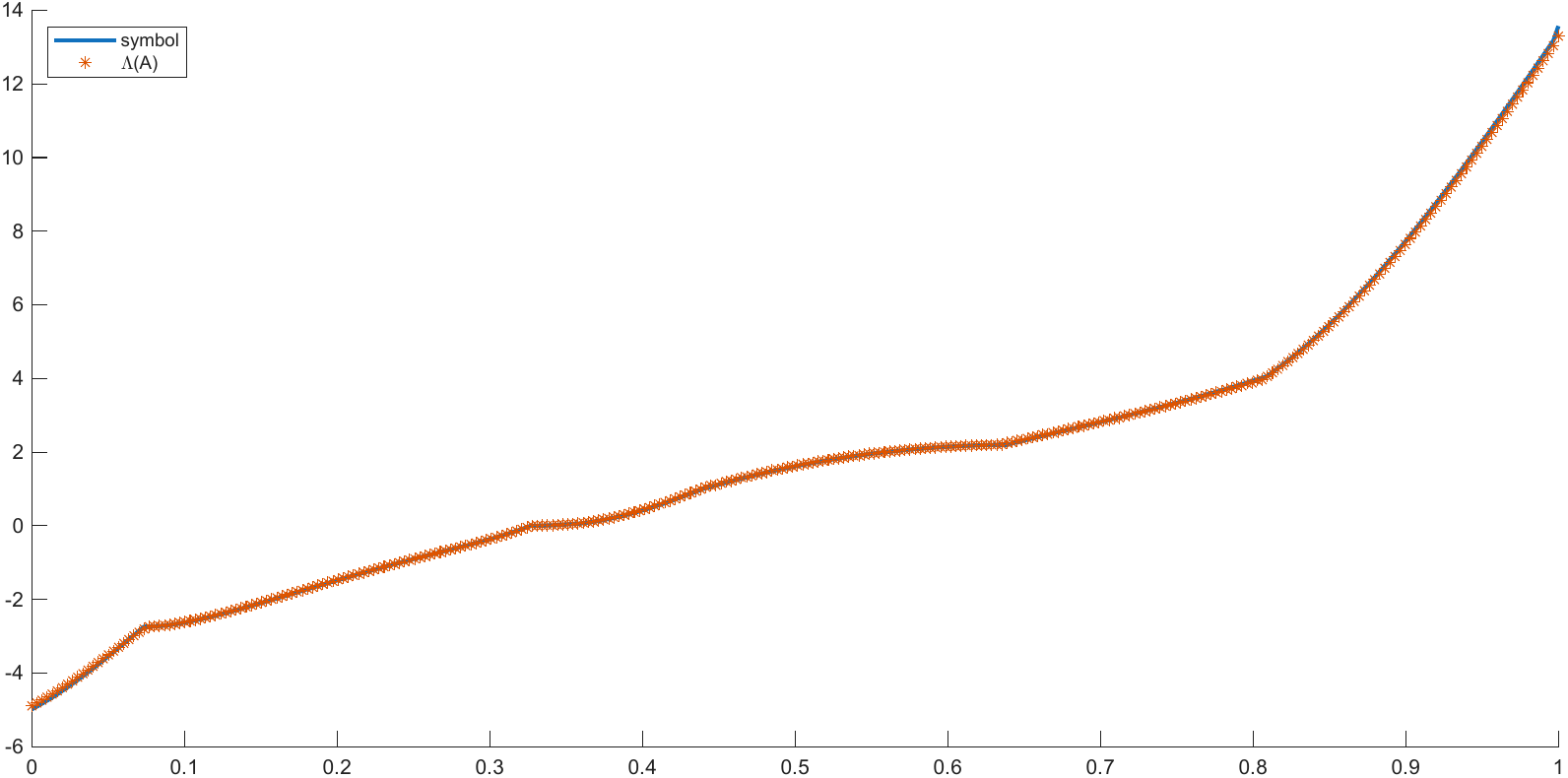}%
    \qquad \qquad 
    \includegraphics[width=70mm]{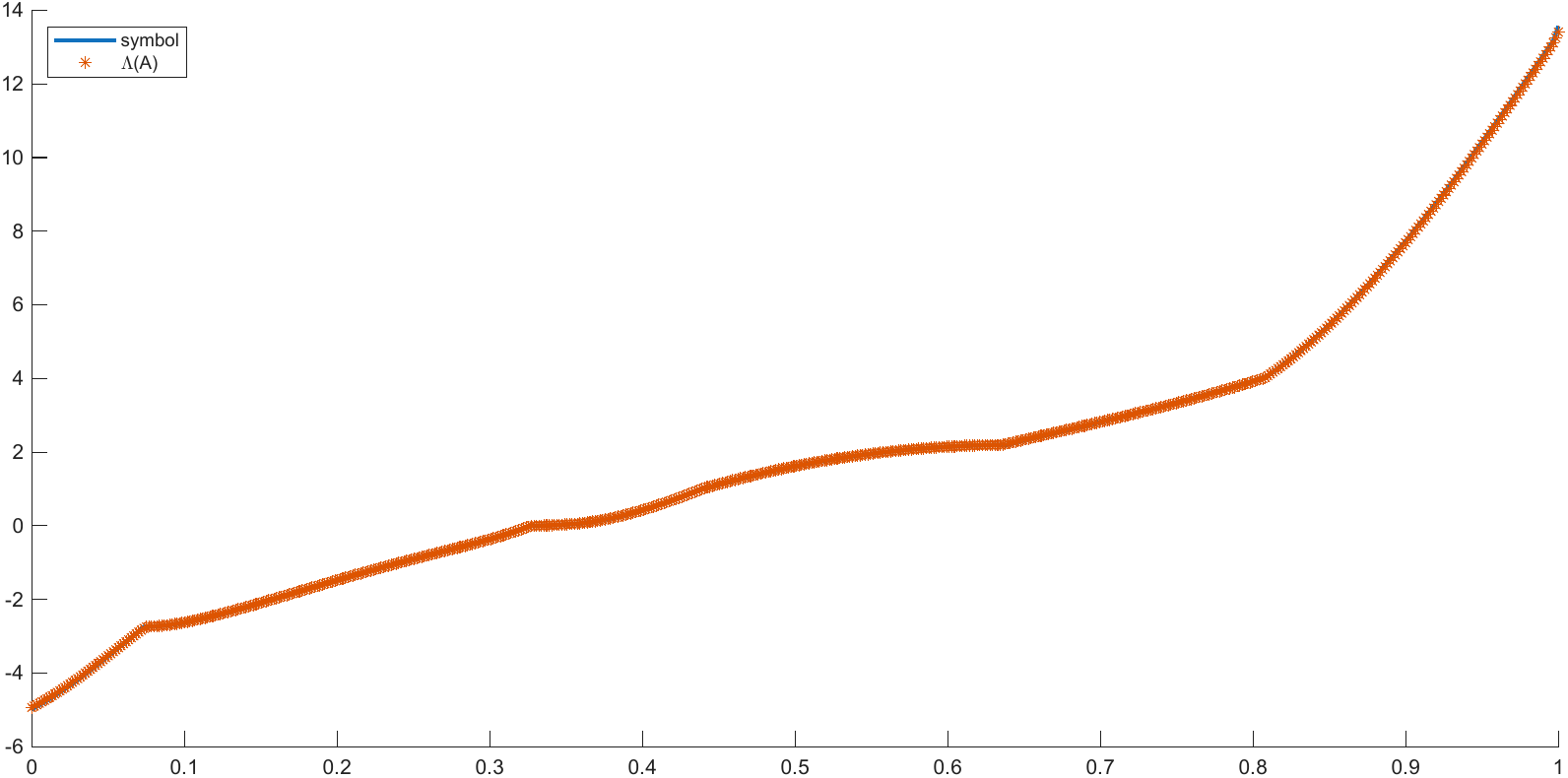}
    \caption{Comparison between the real part of the spectrum of $\mathcal{A}_n$ and a uniform sampling of $F(\theta)$ for $n=100,200$}
    \label{Caso_2_limite}
\end{figure}

\subsubsection{Case $k=3$ without technical assumptions}

In this subsection, we present the final example of the paper, where we consider the case $k=3$ and we abandon the last (technical) assumption $\mathcal{L}\left( \{\theta \in [-\pi,\pi] \, : \, f_{i,j}(\theta) =0\} \right)=0$. In particular, we set $f_{1,1}(\theta)= -\theta^4, f_{1,2}(\theta) = 2-2\cos(\theta), f_{2,1}(\theta)= \theta^2, f_{2,2}(\theta)=1, f_{2,3}(\theta)=\theta \mathds{1}_{[0,\pi]}(\theta), f_{3,2}(\theta)=|\theta|, f_{3,3}(\theta)= \sqrt{|\theta|} $. Based on the numerical evidence contained in this subsection and driven by the analysis in \cite{IKLS25}, we think that it is possible to remove any technical hypothesis on the generating functions, simply asking for positive definiteness of all the blocks (see Conjecture \ref{conj_final} for more details).
However, it is important to underline that, differently from the cases analyzed in the previous subsections, the eigenvalues of $\mathcal{A}_{n}$ possess a complex part substantially greater than before. Nevertheless, when considering the real part of the eigenvalues, we still see a very good agreement between the candidate symbol and the spectrum of $\mathcal{A}_n$, as shown in Figure \ref{Caso_crazy}. Moreover, we also notice that the imaginary part of the eigenvalues of $\mathcal{A}_n$ decreases as $n$ increases as shown in Figure \ref{error}, with order approximately of $O(1/n)$ as can be inferred by Table \ref{table:error}.

\begin{figure}[H]
\centering
    \includegraphics[width=70mm]{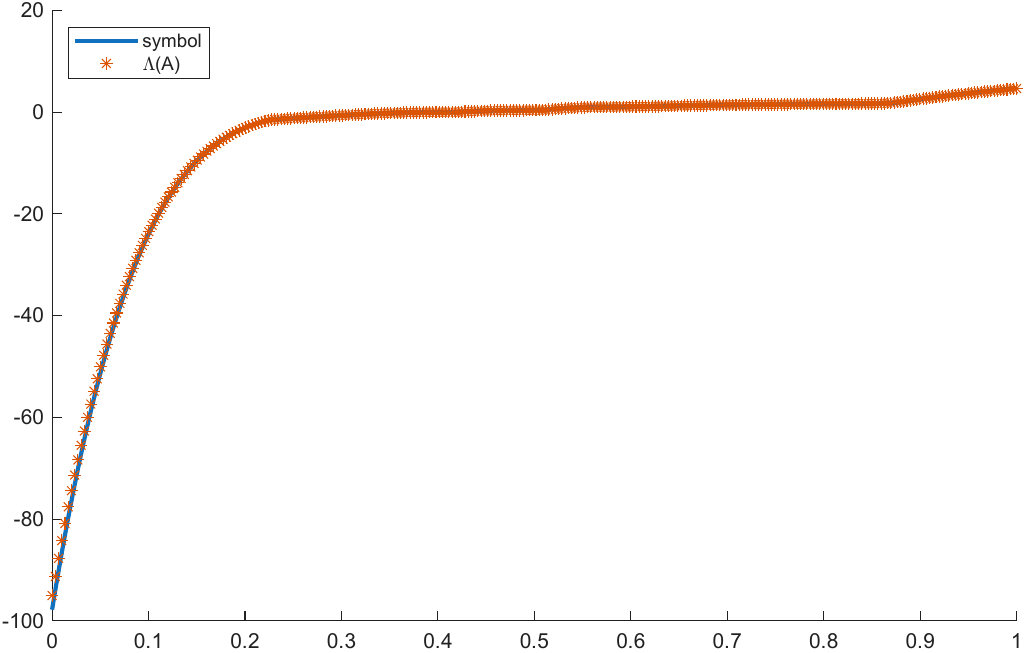}%
    \qquad \qquad 
    \includegraphics[width=70mm]
    {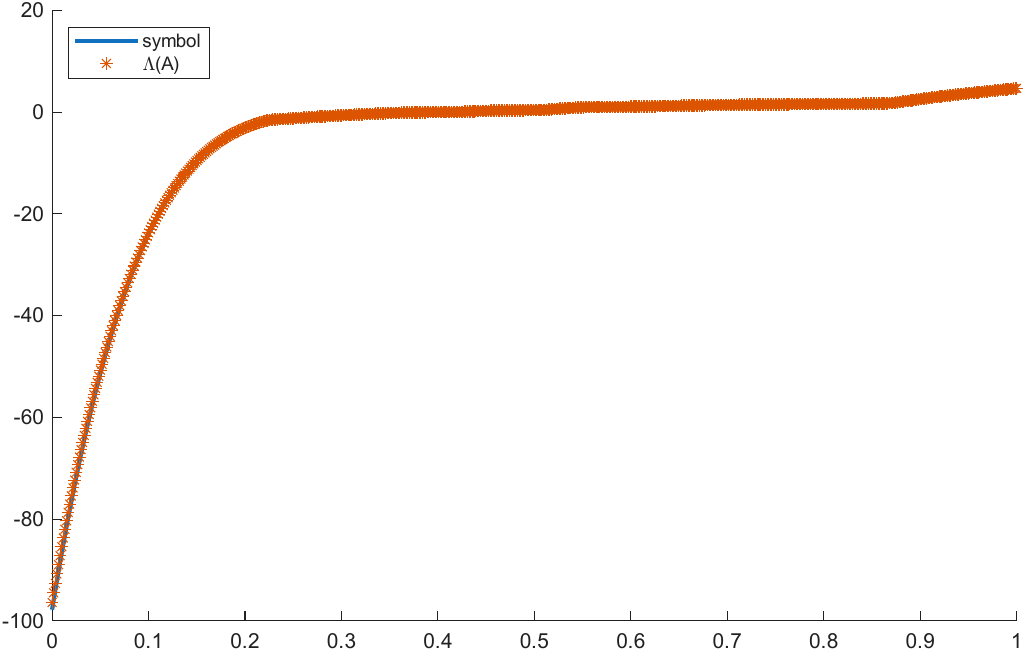}
    \caption{Comparison between the real part of the spectrum of $\mathcal{A}_n$ and a uniform sampling of $F(\theta)$ for $n=100,200$}
    \label{Caso_crazy}
\end{figure}

\begin{figure}[H]
\centering
    \includegraphics[width=140mm]
    {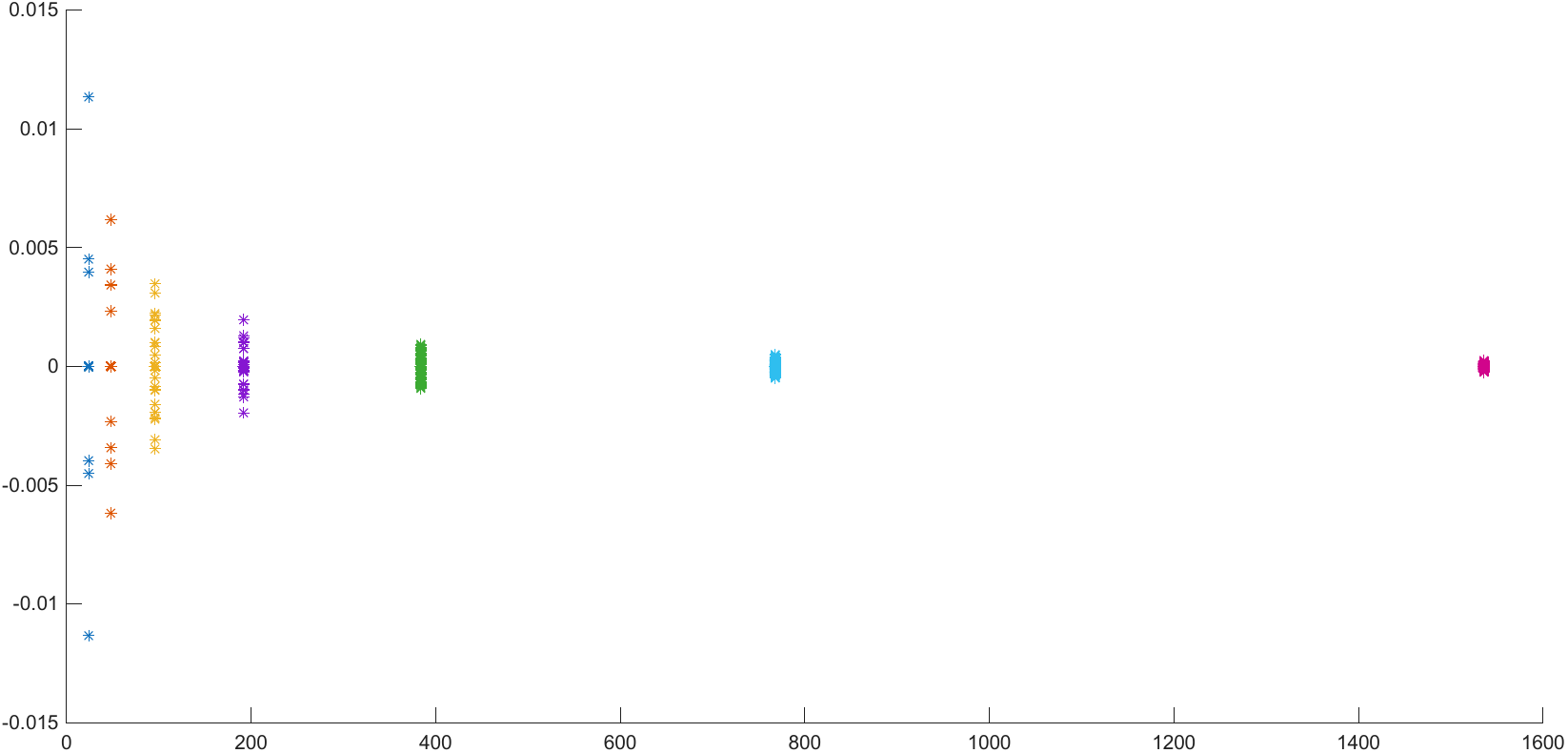}
    \caption{In each column, depending on $n$, it is represented the value of the imaginary part of all the eigenvalues of $\mathcal{A}_n$}
    \label{error}
\end{figure}

\begin{table}[h!]
\centering
\begin{tabular}{|c|c|}
\hline
$n$ & $\max_{\lambda} \left|\Im(\lambda(\mathcal{A}_n))\right|$ \\
\hline
24 & 0.0113  \\
\hline
48 & 0.0062  \\
\hline
96 & 0.0035 \\
\hline
192 & 0.0020 \\
\hline
384 &  9.3026e-04\\
\hline
768 & 4.9599e-04\\
\hline
1536 & 2.4318e-04 \\
\hline
\end{tabular}
\caption{Maximum over $\lambda$ of the absolute value of the imaginary part of $\lambda(\mathcal{A}_n)$, varying $n$}
\label{table:error}
\end{table}

\subsection{Preconditioning strategy}

Consider the linear system 
\begin{equation}\label{eq_system}
    \mathcal{A}_n x_n= b_n,
\end{equation}
with $b_n \in \mathbb{C}^{d_n}$ and $\mathcal{A}_n \in \mathbb{C}^{d_n \times d_n}$ as in \eqref{eq:A_Toeplitz}. Studying under which conditions on the generating functions $f_{i,j}$ the matrix $\mathcal{A}_n$ is invertible is a problem of its own and will be the subject of future research. Assuming this is the case and inspired by the circulant approximation of our analysis, we propose the following preconditioner to solve Equation \eqref{eq_system}:
\begin{equation*} \label{eq:preconditioner}
    \mathcal{C}_n = 
    \begin{bmatrix}
C_n(f_{1,1}) & C_{n}(f_{1,2}) &  &  \\
C_n(f_{2,1}) & C_n(f_{2,2}) & \ddots &  \\
 & \ddots & \ddots & C_n(f_{k-1 , k}) \\
 &  & C_n(f_{k, k-1}) & C_n(f_{k, k})
\end{bmatrix},
\end{equation*}
where again, {$C_{n}(f_{i,j}) \in \mathbb{C}^{n \times n}$, with $n = d_n / k$,} is as in Definition \ref{def_approx_circ}. Now, denote by
\begin{equation*}
    F_n=\frac{1}{\sqrt{n}} \left[\mathrm{e}^{-\i \frac{2\pi (j-1)(k-1)}{n}}\right]_{j,k=1}^{n}
\end{equation*}
the usual discrete Fourier transform, so that $C_n(f)=F_n \diag (\sqrt{n} F_n^{*} \bfc_f) F_n^{*}$, with
$$
\bfc_f = (c_1(f),\dots,c_n(f)), 
$$
where, for $k=1,\dots,n$,
\begin{equation*}
    c_k(f)= \frac{(n-k+1)\tilde{f}_{k-1}+(k-1) \tilde{f}_{-n+k-1}}{n}.
\end{equation*}

A fast procedure for solving a system of the type $ \mathcal{C}_{n} y = z $ can be easily obtained. 
Let $D_{ s,t } =  \diag(\tilde{\bfc}^{(s,t)}) \in \mathbb{C}^{n \times n }$, where $ \tilde{\bfc}^{(s,t)} = \sqrt{n} F_n^{*} \bfc_{f_{s,t}}$ and define
\begin{equation*}
   \mathcal{D}_{n} = 
    \begin{bmatrix}
    D_{1, 1} & D_{1, 2} &  &  \\
    D_{2, 1} & D_{2, 2} & \ddots &  \\
     & \ddots & \ddots & D_{k-1, k} \\
     &  & D_{k, k-1} & D_{k, k} 
    \end{bmatrix} \in \mathbb{C}^{ d_n \times d_n},
\end{equation*}
so that 
\begin{equation*}
    \mathcal{C}_{ n } = \left( I_{ k } \otimes F_{ n } \right) \mathcal{D}_{n} \left( I_{ k } \otimes F_{ n} \right)^*.
\end{equation*}
Moreover, due to structure of $\mathcal{D}_n$, there exists a permutation $P_n$ of size $d_n \times d_n$ such that $P_n \mathcal{D}_n P_n^T$ is a block diagonal matrix whose $k \times k$ blocks exhibit a tridiagonal structure:
\begin{equation*}
    P_n \mathcal{D}_n P_n^T = \diag_{i = 1, \dots, n} \left( T_{i} \right),
\end{equation*}
where 
\begin{equation*}
    T_{ i} = \left[ \left( D_{s, t }\right)_{i, i} \right]_{s, t = 1}^{ k} =
    \begin{bmatrix}
\tilde{\bfc}^{ \left(1, 1 \right)}_i & \tilde{\bfc}^{ \left(1, 2 \right)}_{i} &  &  &  &  &  \\
\tilde{\bfc}^{ \left(2, 1 \right)}_{i} & \tilde{\bfc}^{ \left(2,2 \right)}_i & \ddots &  &  &  &  \\
 &  &  &  & \ddots & \ddots & \tilde{\bfc}^{ \left(k-1, k \right)}_{i} \\
 &  &  &  &  & \tilde{\bfc}^{ \left(k, k-1 \right)}_{i} & \tilde{\bfc}^{ \left(k, k\right)}_i .
\end{bmatrix}
\end{equation*}
Hence, after defining $\mathcal{T}_n =\diag_{i = 1, \dots, n} \left( T_{i} \right) $ we have
\begin{equation*}
    \mathcal{C}_{n} = \left( I_{ k } \otimes F_{ n } \right) P_n^T \mathcal{T}_n P_n \left( I_{ k } \otimes F_{ n } \right)^*,
\end{equation*}
so that the linear system $ \mathcal{C}_{n} y = z $ is equivalent to computing
\begin{equation*} \label{eq:costo_prec}
    y = \left( I_{ k } \otimes F_{ n } \right) P_n^T \mathcal{T}_n^{-1} P_n \left( I_{ k } \otimes F_{ n } \right)^* z.
\end{equation*}
Products of the type $\left( I_{ k } \otimes F_{ n } \right) x$ and $\left( I_{ k } \otimes F_{ n } \right)^* x$ can be performed in $\mathcal{O} \left( kn \log_2 n \right)$ by leveraging $k$ independent Fast Fourier Transforms; moreover, the inversion of $\mathcal{T}_n$ requires $ \mathcal{O} \left( n k \right) $ operations by exploiting its block tridiagonal structure with Thomas Algorithm, for a total computational cost of $2 \mathcal{O} \left( kn \log_2 n \right) + \mathcal{O} \left( n k \right) $. 
Finally, by recalling that $ k $ is fixed, such cost reduces to $\mathcal{O} \left( n \log_2 n \right)$.

\subsection{Numerical validation of the proposed preconditioner}
In this section, we estimate the performance of the preconditioner \eqref{eq:preconditioner} when solving a linear system $\mathcal{A}_n x_n = b_n$, with $\mathcal{A}_n$ as in \eqref{eq:A_Toeplitz} via the GMRES method. 
{It is known (see \cite{EM22, CA96}) that the convergence rate of the method relates to both the condition number $\kappa_2 (A_n)$ and the clustering properties of its eigenvalues.} \\
Let $k = 2$ and set, for $ \theta \in \left[ -\pi, \pi \right]$,
\begin{equation*}
    \begin{cases}
         f_{1,1} \left( \theta \right) =2-2 \cos \left( \theta \right) \\
        f_{1,2} \left( \theta \right)= \left(2-2 \cos \left( \theta \right) \right)^2+1 \\
        f_{2,1}\left( \theta \right)=1 \\
        f_{2,2}\left( \theta \right)= \theta^2;
    \end{cases}
\end{equation*}
notice that such choice yields to a dense $T_{n} \left(f_{2,2} \right)$ block. As a validation benchmark for the preconditioner's performance we select the iteration count for the GMRES method without restart and the tolerance set as $10^{-6}$ as $n$ increases; moreover, we do not impose a bound for the iteration number. The results are summarized in Table \ref{tab:iterazioni}; notice that the iteration count for the problem without preconditioning exhibits a growth as the size increases, while for its preconditioned version the number of iterations plateaus at $6$. 

\begin{table}[H] 
\begin{tabular}{|l|l|l|}
\hline
$n$ & non-preconditioned iteration count & preconditioned iteration count \\  
\hline
$100$  & $135$  & $7$ \\
\hline
$200$  & $277$  & $6$ \\
\hline
$300$  & $422$  & $6$ \\
\hline
$400$  & $569$  & $7$ \\
\hline
$500$ & $712$  & $6$ \\
\hline
$600$ & $859$  & $6$ \\
\hline
$700$ & $1017$ & $7$ \\
\hline
$800$ & $1160$ & $6$ \\
\hline
$900$ & $1310$ & $4$ \\
\hline
$1000$ & $1465$ & $6$ \\
\hline
$1100$ & $1612$ & $6$ \\
\hline
$1200$ & $1763$ & $6$ \\
\hline
$1300$ & $1917$ & $6$ \\
\hline
$1400$ & $2067$ & $6$ \\
\hline
$1500$ & $2216$ & $6$ \\
\hline
$1600$ & $2370$ & $6$ \\
\hline
$1700$ & $2524$ & $6$ \\
\hline
$1800$ & $2672$ & $6$ \\
\hline
$1900$ & $2824$ & $6$ \\
\hline
$2000$ & $2981$ & $6$ \\
\hline
\end{tabular}
\caption{Comparison between GMRES without restart iteration count for the non preconditioned and preconditoned system, with tolerance set as $10^{-6}$ and the right side vector equal to $\left[1, \dots, 1\right]^T \in \mathbb{C}^{d_{n}}$; note that $n$ denotes the size of each block and not the global matrix dimension.}
\label{tab:iterazioni}
\end{table}

\section{Conclusions}\label{sec:end}

In the present paper, we studied the spectral distribution of matrix-sequences with a non-Hermitian block structure, obtaining the result that the GLT symbol well approximates the spectrum of the sequence, under the hypothesis that each block has a Toeplitz structure generated by an essentially positive bounded function.\\
The result is based on the approximation, in the Frobenius norm, of each Toeplitz block with matrices in the circulant algebra.\\
A number of questions arises from the present contribution in terms of theoretical foundations of the results and possible applications and extensions to different settings. Here we list and discuss some of them.
\begin{enumerate}
    \item We extensively used the fact that whenever we consider the geometric mean of two sequences, both generating functions have strictly positive essential infimum. This is a technical assumption that allows us to control the behavior of the spectral norm of the sequence of geometric means. However, as our numerical evidence shows, we can probably eliminate this hypothesis for the more natural assumption that each block is simply Hermitian positive definite.
    \item We analyze the problem in the very specific setting of Toeplitz blocks generated by uni-variate scalar symbols. Clearly, it would be interesting to extend this setting in different directions by considering first multi-variate scalar symbols and then uni-variate and multi-variate matrix-valued symbols for each block.
    \item In \cite{AFGS26,AGS26}, the authors studied the singular value and spectral distributions (the latter in the Hermitian case) even when the blocks do not have the same dimension $n$. Future research could consider under which assumptions this can be extended in our non-Hermitian setting.
    \item The theory of spectral distribution for Toeplitz sequences is usually developed in the case of generating functions in $L^{1}([-\pi,\pi])$. Here, we assumed that the functions are in $L^{\infty}([-\pi,\pi])$ in order to obtain control over the spectral norm of the sequence of geometric means. It would certainly be of interest to study whether this assumption may be relaxed to include all symbols in $L^{1}([-\pi,\pi])$. Regarding this, since our proofs rely on the circulant approximation of Toeplitz matrices in the Frobenius norm, we think that a natural extension of our results can be studied for symbols in $L^2([-\pi,\pi])$.
    \item An interesting topic for future research is represented by the possibility of further exploring applications and numerical methods for solving large linear systems involving matrices with a similar non-Hermitian structure. Note that, in this paper, we already show how to take advantage of the structure of the sequence and the knowledge of the analytical features of the symbol to improve the convergence of GMRES. Moreover, in terms of total computational cost for the solution of linear systems, it would be interesting to explore if it is possible to find even more performing preconditioners under special conditions, e.g., in the case of polynomial symbols. In this case, we would expect to find the solution in linear computational time, without having to rely on FFT.
    \item It is known that a matrix $A=(a_{i,j})$ can be symmetrized using a diagonal matrix if the following two conditions are met:
    \begin{enumerate}
        \item $a_{ij} \neq 0 $ if and only if $ a_{ji} \neq 0$ for every $i,j$;
        \item for every cycle $i_1,\dots,i_k,i_{k+1}$ it holds
        \begin{equation*}
            \prod_{l=1}^{k} a_{i_l}, a_{i_{l+1}} = \prod_{l=1}^{k} a_{i_{l+1}}, a_{i_{l}}, 
        \end{equation*}
        where, by convention, $i_{k+1}=i_1$.
    \end{enumerate}
    Notice that, in our case, these assumptions are trivially satisfied by the entries of the matrix-valued symbol $F(\theta)$ for (a.e.) $\theta \in [-\pi,\pi]$. This observation and the discussion in the previous points lead us to propose the following conjecture. 
\begin{conjecture}\label{conj_final}
Let
\begin{equation*}
    \mathcal{A}_n = 
    \begin{bmatrix}
T_n(f_{1,1}) & T_{n}(f_{1,2}) & \dots & T_n(f_{1,k})  \\
T_n(f_{2,1}) & T_n(f_{2,2}) & \ddots & \vdots  \\
 \vdots & \ddots & \ddots & T_n(f_{k-1 , k}) \\
 T_n(f_{k,1}) & \dots & T_n(f_{k, k-1}) & T_n(f_{k, k})
\end{bmatrix},
\end{equation*}
where, for every $i,j=1,\dots, k,$ $f_{i,j} \in L^{2}([-\pi,\pi])$ and real-valued almost everywhere with $T_n(f_{i,j}) $ Hermitian positive definite whenever $i \neq j$, or $f_{i,j}=0$ almost everywhere. Assume that, for every cycle $i_1,\dots,i_k,i_{k+1}$, we have
\begin{equation*}
\prod_{l=1}^{k} f_{i_l,i_{l+1}} (\theta) = \prod_{l=1}^{k} f_{i_{l+1},i_l} (\theta)
\end{equation*}
for a.e. $\theta \in [-\pi,\pi]$, where $i_{k+1}=i_1$.
Then,
\begin{equation*}
 \{\mathcal{A}_n\}_n \sim_{\lambda}  \begin{bmatrix}
f_{1,1} & f_{1,2} & \dots &  f_{1,k}\\
f_{2,1} & f_{2,2} & \ddots & \vdots \\
\vdots & \ddots & \ddots & f_{k-1 , k} \\
f_{k,1} & \dots & f_{k, k-1} & f_{k, k}
\end{bmatrix}.
\end{equation*}

\end{conjecture}
\end{enumerate}

\end{document}